\documentclass[11pt]{amsart}
\usepackage{xcolor,amsmath,amssymb,graphicx}
\usepackage[english]{babel}
\usepackage[T1]{fontenc}
\usepackage{amsmath, amscd,amsthm, amsfonts, amssymb}
\usepackage{hhline, latexsym}
\usepackage{float}

\newlength{\dinwidth}
\newlength{\dinmargin}
\newcommand{\M}{\mathcal{M}}
\newcommand{\R}{\mathbb{R}}
\newcommand{\N}{\mathbb{N}}
\newcommand{\C}{\mathbb{C}}
\newcommand{\Z}{\mathbb{Z}}
\newcommand{\B}{\mathbb{B}}

\newcommand{\Rs}{\mathcal{R}}

\newtheorem{definition}{Definition}[section]
\newtheorem{theorem}{Theorem}[section]
\newtheorem{proposition}{Proposition}[section]

\newtheorem{remark}{Remark}[section]
\newtheorem{lemma}{Lemma}[section]
\newtheorem{example}{Example}[section]

\begin{document}

\def\theequation {\thesection.\arabic{equation}}
\makeatletter\@addtoreset {equation}{section}\makeatother

% Definition of title page:
\title[]{Computation of the topological type of a real Riemann surface} 

\author{C.~Kalla}
\address{	 Centre de recherches math\'ematiques
Universit\'e de Montr\'eal, Case postale 6128, 
Montr\'eal H3C 3J7, Canada}
    \email{kalla@crm.umontreal.ca}

\author{C.~Klein}
\address{Institut de Math\'ematiques de Bourgogne,
		Universit\'e de Bourgogne, 9 avenue Alain Savary, 21078 Dijon
		Cedex, France}
    \email{Christian.Klein@u-bourgogne.fr}

\date{\today}    % optional

\begin{abstract}

We present an algorithm for the computation of the topological type of a real 
compact Riemann surface associated to an algebraic curve, i.e., its genus and the properties 
of the set  of fixed points of the anti-holomorphic involution $\tau$, namely, the number of its connected 
components, and whether this set  divides the surface into one  or 
two connected components. This is achieved by transforming an arbitrary 
canonical homology basis to a homology basis where the $\mathcal{A}$-cycles are invariant under the  anti-holomorphic involution $\tau$.

\end{abstract}


\thanks{We thank V.~Shramchenko for useful discussions and hints. 
This work has been supported in part by the project FroM-PDE funded by the European
Research Council through the Advanced Investigator Grant Scheme,  and the ANR via the program ANR-09-BLAN-0117-01. }

\maketitle

\section{Introduction}

Riemann surfaces have many applications in physics and mathematics as 
in topological field theories and in the theory of integrable 
partial differential equations (PDEs). In concrete applications such as 
solutions of PDEs, e.g.\ Korteweg-de Vries and nonlinear Schr\"odinger 
(NLS) equations, see e.g.\ \cite{BBEIM} and references therein, 
physical  quantities as for instance the amplitude 
of a water wave are real. Thus 
reality conditions on the solutions are important in practice. The corresponding solutions have 
to be constructed on real Riemann surfaces, i.e., surfaces with an 
anti-holomorphic 
involution $\tau$ acting as the complex conjugation on a local parameter 
on the surface. Regularity conditions for these solutions depend on 
the topological type of the surface, i.e., whether there are 
connected sets of
fixed points of the involution $\tau$, the \emph{real ovals}, and 
whether these ovals separate the surface into two connected components.

It is well known that all compact Riemann surfaces can be realized via
nonsingular algebraic curves in $\mathbb{P}^{2}(\C)$.
F. Klein \cite{Klein} observed that a real Riemann surface can be 
obtained in an analogous way from a nonsingular real plane algebraic curve $\Rs$ with an affine part of the form 
\begin{equation}
   f(x,y)= \sum_{n=1}^{N}\sum_{m=1}^{M}a_{mn}x^{m}y^{n}=0, \qquad     
   x,y\in \mathbb{C},\quad a_{mn}\in \mathbb{R}.
    \label{algcur}
\end{equation}
 The focus of this 
paper is on real compact Riemann surfaces.
For curves of the form (\ref{algcur}) the action of the complex 
conjugation gives rise to an antiholomorphic involution 
$\tau$ defined on $\Rs$ by $\tau(x,y)=(\bar{x},\bar{y})$. The 
set of fixed points of $\tau$ is denoted by $\Rs(\R)$ and is called 
the real part of $\Rs$. The connected components of $\Rs(\R)$ are 
called real ovals. Historically, the first result in the topology of real 
algebraic curves was obtained by Harnack \cite{Harnack}: the 
number $k$ of real ovals of a curve $\Rs$ of genus $g$ satisfies 
$0\leq k \leq g+1$.  In other words, the number of connected 
components of the real part of a real nonsingular plane algebraic 
curve cannot exceed $g+1$.           
Curves with the maximal number of real ovals are called 
\emph{M-curves}. 

%An example (\ref{trott}) for the real ovals of an M-curve of 
%degree 3 and genus 3, the Trott curve \cite{Trott}, can be seen in 
%Fig.~\ref{figtrott}.
%\begin{figure}[htb!]
%\begin{center}
%\includegraphics[width=0.7\textwidth]{realtrott.pdf}
%\end{center}
% \caption{Real ovals of the Trott curve (\ref{trott}), an M-curve of 
% genus 3.}
% \label{figtrott}
%\end{figure}

The complement $\mathcal{R}\setminus \mathcal{R}(\R)$ has either one 
or two connected components: if $\mathcal{R}\setminus\mathcal{R}(\R)$ 
has two components, the curve $\mathcal{R}$ is called a 
\textit{dividing} curve,  otherwise it is called 
\textit{non-dividing} (notice that an M-curve is always a dividing 
curve). The topological type of $\Rs$ is usually denoted by $(g,k,a)$ 
where $g$ denotes the genus, $k$ the number of real ovals, and $a=0$ 
if the curve is dividing, $a=1$ if it is non-dividing.     This 
implies that the topological type of a  curve without real oval is 
$(g,0,1)$.   
Notice that the first part of Hilbert's 16th problem is concerned 
with the relative configuration of
real ovals of a plane algebraic curve of given degree in 
$\mathbb{P}^{2}(\R)$, i.e., how many ovals can lie in interior of another 
oval. This question has been studied by many 
authors, see, for instance, \cite{Hilbert1, Rohn, Pet, Gud, Viro, 
Arnold} and references therein. However, until now the complete answer 
is known only for curves of degree 7 and less. We will not  discuss 
this topic here, but we would like to mention that in general, a solution to this 
problem, namely, the knowledge of the embedding $\Rs(\R)\subset 
\mathbb{P}^{2}(\R)$, does not provide any  information on the  
embedding $\Rs(\R)\subset \Rs$ which is the subject of the present paper. For instance, it is possible to 
construct two real plane algebraic curves having the same degree and 
the same configuration of ovals, one of them being dividing and the 
other non-dividing (see \cite{Gab} p.~8).

The aim of this paper is twofold: to determine the topological type $(g,k,a)$ 
of a real algebraic curve of the form (\ref{algcur}) with a numerical 
approach, and to transform  periods of the 
holomorphic differentials of the curve to a form where the 
$\mathcal{A}$-periods are real. 
There exist various 
algorithms which give the oval arrangements of a given real algebraic 
curve, see, for instance, \cite{Arnon, Coste, Sak, Feng, Gonz,  Hong, 
Seidel}, all of them  following the same scheme.  But to the 
best of our knowledge, there exists no algorithm that computes the parameter $a$ in the 
topological type $(g,k,a)$ of $\Rs$, which encodes the property of 
the curve to be dividing or not.  The starting point of our algorithm 
is the work 
by Deconinck and van Hoeij who                 
developed an approach to the symbolic-numerical treatment of  
algebraic curves. This approach is distributed as the 
\textit{algcurves} package with Maple, see \cite{deco1,deco2,deho}. A purely numerical approach to real 
hyperelliptic Riemann surfaces was given in \cite{FK1,FK2}, and for general 
Riemann surfaces in \cite{FK}. For a review on computational 
approaches to Riemann surfaces the reader is referred to \cite{Bob}. 

The codes \cite{deho,FK} compute the periods of a Riemann surface in 
a homology basis which is determined by an algorithm due to Tretkoff 
and Tretkoff \cite{tretalg}. This homology basis is in general not 
adapted to possible symmetries of the curve (as the involution $\tau$ 
of real curves). It means that the action on the computed homology basis of any automorphism of the curve  cannot be expressed  in a simple way 
in terms of this basis. 
However, the choice of a basis, where certain cycles are invariant 
under the automorphisms, is 
often convenient in applications. In the context of 
solutions to integrable PDEs on general compact Riemann surfaces as 
for the Kadomtsev-Petviashvili (KP) (see \cite{DN}) and the 
Davey-Stewartson (DS) equations (see \cite{Mal,Kalla}), 
smoothness conditions are formulated conveniently in a homology basis 
adapted to the anti-holomorphic involution $\tau$ defined on the 
surface. For instance, on a real surface there exists a canonical homology basis $(\mathcal{A},\mathcal{B})$ (that we call for 
simplicity \emph{symmetric homology basis} in the following) satisfying the conditions  (see \cite{sesi,Vin})
\begin{equation*}
\left(\begin{matrix}
\tau \mathbf{\mathcal{A}}\\
\tau \mathbf{\mathcal{B}}
\end{matrix}\right)
=
\left(\begin{matrix}
\mathbb{I}_{g}&\,\,0\\
\mathbb{H}\,\,&-\mathbb{I}_{g}
\end{matrix}\right)
\left(\begin{matrix}
\mathbf{\mathcal{A}}\\
\mathbf{\mathcal{B}}
\end{matrix}\right), 
\end{equation*}
where $\mathbb{H}\in\mathcal{M}_{g}(\Z/2\Z)$ is a  $g\times g$ matrix which 
depends on the topological type $(g,k,a)$ (see section 2); here $\mathbb{I}_{g}$ denotes the $g \times g$ unit matrix.

In \cite{KKnum} we studied for the 
first time numerically solutions to integrable equations from  
the family of NLS equations, namely, the multi-component nonlinear Schrödinger equation
and the $(2+1)$-dimensional DS equations. 
A symplectic
transformation of the computed homology basis 
to the symmetric homology basis  
was introduced in \cite{KKnum} and was constructed explicitly for concrete examples. 
In the present work we give in Theorem \ref{Theorem} a complete description of such a symplectic  transformation   
depending  on the topological type of the underlying real algebraic 
curve. The formulas are expressed in terms of period matrices of 
holomorphic differentials on the curve; these holomorphic 
differentials must satisfy the condition 
$\overline{\tau^{*}\nu_{j}}=\nu_{j}$, $j=1,\ldots,g$, where  
$\tau^{*}$ is the action of $\tau$ lifted to the space of 
holomorphic differentials.   This allows to give an 
algorithm to construct explicitly a symplectic transformation of an 
arbitrary canonical homology basis for a real Riemann surface to the 
symmetric form. The algorithm permits to systematically study smooth 
real solutions on general real Riemann surfaces for equations like 
KP and DS starting from a representation of the surface via an algebraic 
curve which so far was only possible for hyperelliptic surfaces. 
In addition, this provides a numerical way to compute the 
topological type of a real Riemann surface for given periods of the 
holomorphic differentials.

The paper is organized as follows. In section 2 we introduce a 
symmetric homology basis which depends on the topological type of 
$\Rs$. In section 3 we give explicitly the symplectic transformation 
between the computed homology basis and the symmetric homology basis. 
This result will be used in section 4 to construct an algorithm which 
gives the topological type $(g,k,a)$ of a real algebraic curve for 
given periods of the holomorphic differentials. In 
section 5 we discuss examples of real curves for higher genus. Some 
concluding remarks are added in section 6.

\section{Symmetric homology basis}

In what follows $\Rs$ denotes a compact Riemann surface of genus $g>0$.
A homology basis $(\mathcal{A},\mathcal{B}):=(\mathcal{A}_{1},\ldots,\mathcal{A}_{g},\mathcal{B}_{1},\ldots,\mathcal{B}_{g})$  with the following intersection indices  
\[\mathcal{A}_{i}\circ\mathcal{B}_{j} = \delta_{ij} \qquad \quad\mathcal{A}_{i}\circ\mathcal{A}_{j}=\mathcal{B}_{i}\circ\mathcal{B}_{j}=0,\]
is called a canonical basis of cycles. 
% If the Riemann surface $\Rs$ is cut along this basis, one obtains a $4g$-gon where each cycle corresponds to a pair of the sides $\mathcal{A}_{i},\mathcal{A}_{i}^{-1},\mathcal{B}_{i},\mathcal{B}_{i}^{-1}$ which are identified on $\Rs$.
With $\mathcal{A}$ (resp. $\mathcal{B}$) we denote the vector $(\mathcal{A}_{1},\ldots,\mathcal{A}_{g})^{t}$ (resp. $(\mathcal{B}_{1},\ldots,\mathcal{B}_{g})^{t}$).
Canonical homology bases are related via a  symplectic transformation. 
Let $(\mathbf{\mathcal{A}},\mathbf{\mathcal{B}})$ and 
$(\mathbf{\tilde{\mathcal{A}}},\mathbf{\tilde{\mathcal{B}}})$ be  arbitrary canonical homology bases  on $\Rs$. Then there exists a symplectic matrix $\left(\begin{matrix}
A&B\\
C&D
\end{matrix}\right)\in Sp(2g,\Z)$ such that
\begin{equation}
\left(\begin{matrix}
A&B\\
C&D
\end{matrix} \right)
\left(\begin{matrix}
\mathbf{\tilde{\mathcal{A}}}\\
\mathbf{\tilde{\mathcal{B}}}
\end{matrix} \right) 
=
\left(\begin{matrix}
\mathbf{\mathcal{A}}\\
\mathbf{\mathcal{B}}
\end{matrix} \right).  \label{transf Vinn}
\end{equation}
Recall that a symplectic matrix $M\in Sp(2g,\Z)$
 satisfies $M^{t}J_{g}M=J_{g}$, with the matrix $J_{g}$ given by
$J_{g}=\left(\begin{matrix}
0&\,\,\mathbb{I}_{g}\\
-\mathbb{I}_{g}\,\,&0
\end{matrix}\right),$ 
where $\mathbb{I}_{g}$ denotes the $g \times g$ unit matrix. 
Symplectic matrices $M=\left(\begin{matrix}
A&B\\
C&D
\end{matrix}\right)\in Sp(2g,\Z)$ are characterized by the following system:
\begin{align}
A^{t}D-&C^{t}B=\mathbb{I}_{g},  \label{Symp1}\\
A^{t}C&=C^{t}A,  \label{Symp2}\\
D^{t}B&=B^{t}D.\label{Symp3}
\end{align}
Moreover, the inverse matrix $M^{-1}$ is given by 
\begin{equation}
M^{-1}=\left(\begin{matrix}
\,\,\,\,D^{t}&-B^{t}\,\\
-C^{t}&\,\,\,A^{t}
\end{matrix}\right).  \label{inverse}
\end{equation}

Now let $\tau$ be an anti-holomorphic involution defined on 
$\mathcal{R}$. 
Recall that $(g,k,a)$ denotes the topological type of $\Rs$, where 
$k$ is the number of connected components of $\mathcal{R}(\R)$ (the 
set of fixed points of $\tau$), and  $a=0$ if the curve is dividing 
(i.e., if $\mathcal{R}\setminus\mathcal{R}(\R)$ 
has two components),  $a=1$ if it is non-dividing (i.e., if $\mathcal{R}\setminus\mathcal{R}(\R)$ 
has just one component). A curve with the topological type $(g,g+1,0)$ is called an M-curve. 
According to Proposition 2.2 in  Vinnikov's paper 
\cite{Vin}, there exists a canonical homology basis 
$(\mathcal{A},\mathcal{B})$ (called for 
simplicity
\emph{symmetric homology basis}) such that
\begin{equation}
\left(\begin{matrix}
\tau \mathbf{\mathcal{A}}\\
\tau \mathbf{\mathcal{B}}
\end{matrix}\right)
=
\left(\begin{matrix}
\mathbb{I}_{g}&\,\,0\\
\mathbb{H}\,\,&-\mathbb{I}_{g}
\end{matrix}\right)
\left(\begin{matrix}
\mathbf{\mathcal{A}}\\
\mathbf{\mathcal{B}}
\end{matrix}\right), \label{hom basis}
\end{equation}
where $\mathbb{H}$ is a  block diagonal $g\times g$ matrix which depends on the topological type $(g,k,a)$ of $\Rs$ and is defined as follows:
\\\\
- if $k>0$ and  $a=0$,
\[\mathbb{H}={\left(\begin{matrix}
0&1&&&&&&\\
1&0&&&&&&\\
&&\ddots&&&&&\\
&&&0&1&&&\\
&&&1&0&&&\\
&&&&&0&&\\
&&&&&&\ddots&\\
&&&&&&&0
\end{matrix}\right)},\]
- if $k>0$ and  $a=1$,
\[\mathbb{H}={\left(\begin{matrix}
1&&&&&\\
&\ddots&&&&\\
&&1&&&\\
&&&0&&\\
&&&&\ddots&\\
&&&&&0
\end{matrix}\right)};\]
rank$(\mathbb{H})=g+1-k$ in both cases;
\\\\
- if $k=0$, 
\[\mathbb{H}={\left(\begin{matrix}
0&1&&&\\
1&0&&&\\
&&\ddots&&\\
&&&0&1\\
&&&1&0
\end{matrix}\right)}
\quad \text{for even $g$ or} \quad 
\mathbb{H}={\left(\begin{matrix}
0&1&&&&\\
1&0&&&&\\
&&\ddots&&&\\
&&&0&1&\\
&&&1&0&\\
&&&&&0
\end{matrix}\right)} \mbox{  for odd $g$};\]
rank$(\mathbb{H})=g$ if $g$ is even, rank$(\mathbb{H})=g-1$ if $g$ is odd.

\begin{remark}
\emph{For given information whether there are real ovals ($k>0$) or 
    not ($k=0$), the matrix $\mathbb{H}$ completely encodes the 
    topological type of the real Riemann surface. In \cite{sesi} a 
    different, but equivalent form of $\mathbb{H}$ was used, which 
    has ones only in the antidiagonal, if the rank is equal to the 
    genus, or in a parallel to the antidiagonal for smaller rank. }
\end{remark}

\begin{example}
Consider the hyperelliptic curve of genus $g$ defined by the equation
\begin{equation}
y^{2}= \prod_{i=1}^{2g+2} (x-x_{i}), \label{hyp real1}
\end{equation}
where the branch points $x_{i}\in\R$ are ordered such that 
$x_{1}<\ldots<x_{2g+2}$. On such a curve, we can define 
two anti-holomorphic involutions $\tau_{1}$ and $\tau_{2}$, given 
respectively by 
$\tau_{1}(x,y)=(\overline{x},\overline{y})$ and 
$\tau_{2}(x,y)=(\overline{x},-\overline{y})$. 
Projections of real ovals of $\tau_{1}$  on the $x$-plane 
coincide with  the intervals 
$[x_{2g+2},x_{1}],\ldots,[x_{2g},x_{2g+1}]$, 
whereas projections of real ovals of $\tau_{2}$ on the 
$x$-plane coincide with  the intervals 
$[x_{1},x_{2}],\ldots,[x_{2g+1},x_{2g+2}]$. 
Hence the curve (\ref{hyp real1}) is an M-curve with respect to both 
anti-involutions $\tau_{1}$ and $\tau_{2}$.                

If all $x_{i}$ are non-real and pairwise conjugate, the curve 
has no real ovals with respect to the involution $\tau_{2}$; it is 
dividing for the involution $\tau_{1}$.
\label{example rea hyp}
\end{example}

\section{Symplectic transformation between homology bases}

In this section we construct a symplectic transformation between an arbitrary homology basis 
$(\mathcal{\tilde{A}},\mathcal{\tilde{B}})$ and a symmetric homology basis on a real Riemann surface $(\mathcal{R},\tau)$, where $\tau$ denotes the anti-holomorphic involution defined on $\mathcal{R}$. %In particular, as stated in Remark \ref{}, this gives a solution to the problem which consists to find the symplectic transformation between two arbitrary homology basis on the same real Riemann surface.  
The main result of this paper is contained in Theorem \ref{Theorem} which gives the underlying symplectic matrix $\left(\begin{matrix}
A&B\\
C&D
\end{matrix}\right)$ in terms of the period matrices of holomorphic differentials satisfying the condition (\ref{diff mu}) in the homology basis $(\mathcal{\tilde{A}},\mathcal{\tilde{B}})$.  The key ingredient in this context is the description of the action of $\tau$ on the cycles  $(\mathcal{\tilde{A}},\mathcal{\tilde{B}})$,  given in Proposition  \ref{proposition R} by the integer matrix $\mathbf{R}$. Then Theorem \ref{Theorem} states that the column vectors of the matrix $\left(\begin{matrix}
A&B
\end{matrix}\right)^{t}$  form in fact a \textit{$\Z$-basis} of the integer kernel of the matrix $\mathbf{R}^{t}-\mathbb{I}_{2g}$. The matrix $Q$ in Theorem \ref{Theorem} encodes the degree of freedom in the choice of such a $\Z$-basis. For the ease of the reader, we start recalling some basic facts from the theory of \textit{$\Z$-modules} used to prove Theorem \ref{Theorem}, which differs from the usual linear algebra over vector spaces.

\subsection{Basic facts from the theory of $\Z$-modules}
\label{module}

A \textit{$\Z$-module} or more generally a \textit{$\mathbb{A}$-module} where $\mathbb{A}$ denotes a commutative ring, is a natural generalization of vector spaces where the usual scalar field is replaced by the ring $\mathbb{A}$. For a review on the subject we refer to \cite{Lang}. In what follows we assume that $\mathbb{A}$ is the principal ring $\Z$ and we denote by $\mathcal{M}$ a $\Z$-module.

\begin{definition} 
\begin{enumerate}
\item[]
	\item[i.] $\mathcal{M}$ is of \textit{finite type} if it admits a finite set of generators. 
	\item[ii.] $\mathcal{M}$ is said to be free if there exists a $\Z$-basis, namely, 
a set $\{x_{i}\}_{i\in I}\subset \mathcal{M}$ with $I \subset \N$ such that any element $x\in\mathcal{M}$ can be written uniquely as $x=\sum_{i\in I}\alpha_{i}\,x_{i}$ where the scalars $\alpha_{i}\in\Z$ are non-zero only for a finite number of them.
\end{enumerate}
\end{definition}

The following theorems provide important results in the general theory of modules over a principal ring.

\begin{theorem}
\begin{enumerate}
\item[]
	\item[i.] If $\mathcal{M}$ is free and of finite type then all $\Z$-bases of $\mathcal{M}$ are finite with the same cardinality called the $rank$ of $\mathcal{M}$. 
\item[ii.] Since $\Z$ is integral, the rank of $\mathcal{M}$ equals the dimension of the $\mathbb{Q}$-vector space $\mathbb{S}^{-1}\mathcal{M}$ where $\mathbb{S}=\Z\setminus\{0\}$.  
\item[iii.] A submodule of a free $\Z$-module of finite type of rank $n$ is a free $\Z$-module of finite type of rank $r\leq n$. 
\end{enumerate}\label{theo1}
\end{theorem}

The example we will need in this paper are
submodules of $\Z^{n}$ which are because of the Theorem~\ref{theo1} 
free modules of finite type with rank $r\leq n$.

\begin{remark} \emph{Contrary to the case of vector spaces, if 
    $\mathcal{N}$ is a free submodule of the free module 
    $\mathcal{M}$ with the same rank, this does not imply that 
    $\mathcal{N}$ equals $\mathcal{M}$. Consider the submodule $2\Z$ 
    of $\Z$ as an example.} \label{remark module}
\end{remark}

The following theorem, also called \textit{the theorem of the adapted basis}, gives the classification of modules over the  principal ring $\Z$:

\begin{theorem}
Let $\mathcal{M}$ be a free $\Z$-module of rank $n$ and let $\mathcal{N}$ be a submodule of $\mathcal{M}$. Then there exists a basis $(e_{1},\ldots,e_{n})$ of $\mathcal{M}$ and unique non-zero integers $(p_{1},\ldots,p_{r})$ (with $r\leq n$) such that 
\begin{enumerate}
	\item $(p_{1}e_{1},\ldots,p_{r}e_{r})$ is a $\Z$-basis of $\mathcal{N}$ \vspace{0.1cm}
	\item $p_{1}|p_{2}|\ldots|p_{r}$ (this 
means that $\forall i$, $p_{i}$ divides all $p_{j}$ with $j>i$).
\end{enumerate}  \label{adapt basis}
\end{theorem}

In what follows $GL_{n}(\Z)$ denotes the set of $n\times n$ 
invertible matrices over $\Z$; it is well known that the determinant of these matrices equals $\pm 1$. Notice that if $M$ denotes the matrix formed by the column vectors of a $\Z$-basis of a $\Z$-module of rank $n$, then other bases are given by the column vectors of matrices of the form $MQ$ with $Q\in GL_{n}(\Z)$. Therefore, Theorem \ref{adapt basis} has the following matrix interpretation that we will use in section \ref{matrix}: for any non-zero $m \times n$ matrix $M\in \mathcal{M}_{m,n}(\Z)$ of rank $r$, there 
exist matrices $U\in GL_{m}(\Z)$, $V\in GL_{n}(\Z)$ such that
\begin{equation}
UMV=\left(\begin{matrix}
\text{Diag}(p_{1},\ldots,p_{r})& 0\\
0&0
\end{matrix}\right),   \label{Smith}
\end{equation}
where  
$p_{i}\in\Z\setminus\{0\}$ satisfy $p_{1}|p_{2}|\ldots|p_{r}$  
and where $\text{Diag}(.)$ denotes the diagonal matrix. This is 
called the \emph{Smith normal 
form} of $M$. In section \ref{algo} we will use a well known 
algorithm to compute the Smith normal form.
In particular, this algorithm provides a $\Z$-basis of the integer kernel of $M$ given by the last $n-r$ column vectors of the matrix $V$ in (\ref{Smith}).

\subsection{Action of $\tau$ on an arbitrary homology basis}

We denote by  $(\mathbf{\mathcal{A}},\mathbf{\mathcal{B}})$ a 
symmetric homology basis  (i.e.\ which satisfies (\ref{hom basis})). 
Let $(\nu_{1},\ldots,\nu_{g})$ be a basis of holomorphic differentials such that
\begin{equation}
\overline{\tau^{*}\nu_{j}}=\nu_{j}, \qquad j=1,\ldots,g,\label{diff mu}
\end{equation}
where $\tau^{*}$ is the action of $\tau$ lifted to the space of 
holomorphic differentials: $\tau^* \omega(p) = \omega(\tau p)$ for 
any $p\in \Rs$. 
The matrices $P_{\mathcal{A}}$ and $P_{\mathcal{B}}$ defined by
\begin{equation}
(P_{\mathcal{A}})_{ij}=\int_{\mathcal{A}_{i}}\nu_{j}, \qquad (P_{\mathcal{B}})_{ij}=\int_{\mathcal{B}_{i}}\nu_{j},  \qquad i,j=1,\ldots,g  \label{matrix P}
\end{equation}
are called the matrices of $\mathcal{A}$- and $\mathcal{B}$-periods of the differentials $\nu_{j}$. 
From (\ref{hom basis}) and (\ref{diff mu}) we deduce the action of the complex conjugation on the matrices $P_{\mathcal{A}}$ and  $P_{\mathcal{B}}$:
\begin{equation}
(P_{\mathcal{A}})_{ij}\in\R,  \label{PerA Vin}
\end{equation}
\begin{equation}
\overline{P_{\mathcal{B}}}=-P_{\mathcal{B}}+\mathbb{H}P_{\mathcal{A}}. \label{PerB Vin}
\end{equation}

Denote by 
$(\mathbf{\tilde{\mathcal{A}}},\mathbf{\tilde{\mathcal{B}}})$ an 
arbitrary
homology basis. 
From the symplectic transformation (\ref{transf Vinn}) we obtain the 
following transformation law between the matrices 
$P_{\tilde{\mathcal{A}}},P_{\tilde{\mathcal{B}}}$ and $P_{\mathcal{A}},P_{\mathcal{B}}$ defined in (\ref{matrix P}):
\begin{equation}
\left(\begin{matrix}
A&B\\
C&D
\end{matrix}\right)
\left(\begin{matrix}
P_{\tilde{\mathcal{A}}}\\
P_{\tilde{\mathcal{B}}}
\end{matrix}\right) 
=
\left(\begin{matrix}
P_{\mathcal{A}}\\
P_{\mathcal{B}}
\end{matrix}\right).  \label{transf per}
\end{equation}
Therefore, by (\ref{PerA Vin}) one gets
\begin{align}
A \,\text{Re}\left(P_{\tilde{\mathcal{A}}}\right)+B\, \text{Re}\left(P_{\tilde{\mathcal{B}}}\right)&=P_{\mathcal{A}}  \label{A,B Re}\\
A \,\text{Im}\left(P_{\tilde{\mathcal{A}}}\right)+B \,\text{Im}\left(P_{\tilde{\mathcal{B}}}\right)&=0,\label{A,B Im}
\end{align}
and by (\ref{PerB Vin})
\begin{align}
C \,\text{Re}\left(P_{\tilde{\mathcal{A}}}\right)+D\, \text{Re}\left(P_{\tilde{\mathcal{B}}}\right)&=\frac{1}{2}\,\mathbb{H}P_{\mathcal{A}}  \label{C,D Re}\\
C \,\text{Im}\left(P_{\tilde{\mathcal{A}}}\right)+D \,\text{Im}\left(P_{\tilde{\mathcal{B}}}\right)&=\text{Im}\left(P_{\mathcal{B}}\right).\label{C,D Im}
\end{align}

From (\ref{A,B Re}) and (\ref{C,D Im}) it can be checked that the matrices $A 
\,\text{Re}\left(P_{\tilde{\mathcal{A}}}\right)+B\, 
\text{Re}\left(P_{\tilde{\mathcal{B}}}\right)$ and  $C 
\,\text{Im}\left(P_{\tilde{\mathcal{A}}}\right)+D 
\,\text{Im}\left(P_{\tilde{\mathcal{B}}}\right)$ are invertible (the 
first because $P_{\mathcal{A}}$ is, for the second see \cite{KKnum} for more details).  
The following Lemma proved in \cite{KKnum} shows that it is 
sufficient to know the pairs of matrices $A,B$ or $C,D$ to get the 
full symplectic transformation (\ref{transf per}):
\begin{lemma}
The matrices $A,B,C,D\in\M_{g}(\Z)$ solving (\ref{A,B Re})-(\ref{C,D Im}) satisfy:
\begin{align}
A^{t}&=\text{Im}\left(P_{\tilde{\mathcal{B}}}\right)\left[C 
\,\text{Im}\left(P_{\tilde{\mathcal{A}}}\right)+D \,\text{Im}\left(P_{\tilde{\mathcal{B}}}\right)\right]^{-1} \label{mA}\\
B^{t}&=-\text{Im}\left(P_{\tilde{\mathcal{A}}}\right)\left[C 
\,\text{Im}\left(P_{\tilde{\mathcal{A}}}\right)+D \,\text{Im}\left(P_{\tilde{\mathcal{B}}}\right)\right]^{-1} \label{mB}\\
C^{t}&=\frac{1}{2}\,A^{t}\mathbb{H}-\text{Re}\left(P_{\tilde{\mathcal{B}}}\right)\left[A \,\text{Re}\left(P_{\tilde{\mathcal{A}}}\right)+B\, \text{Re}\left(P_{\tilde{\mathcal{B}}}\right)\right]^{-1}  \label{mC}\\
D^{t}&=\frac{1}{2}\,B^{t}\mathbb{H}+\text{Re}\left(P_{\tilde{\mathcal{A}}}\right)\left[A \,\text{Re}\left(P_{\tilde{\mathcal{A}}}\right)+B\, \text{Re}\left(P_{\tilde{\mathcal{B}}}\right)\right]^{-1}.  \label{mD}
\end{align}  
\end{lemma}

The action of $\tau$ on an arbitrary homology basis $(\mathbf{\tilde{\mathcal{A}}},\mathbf{\tilde{\mathcal{B}}})$ can be written as 
\begin{equation}
\left(\begin{matrix}
\tau\mathbf{\tilde{\mathcal{A}}}\\
\tau\mathbf{\tilde{\mathcal{B}}}
\end{matrix}\right) 
=
\mathbf{R}\left(\begin{matrix}
\mathbf{\tilde{\mathcal{A}}}\\
\mathbf{\tilde{\mathcal{B}}}
\end{matrix}\right),  \label{R1}
\end{equation}
where $\mathbf{R}\in\mathcal{M}_{2g}(\Z)$. In the following proposition, we give an explicit expression for the matrix $\mathbf{R}$ in terms of the period matrices 
$P_{\tilde{\mathcal{A}}}$ and $P_{\tilde{\mathcal{B}}}$ only.

\begin{proposition}
The matrix $\mathbf{R}$ defined in (\ref{R1}) is given by
\begin{equation}
\mathbf{R}=
\left(\begin{matrix}
\left(2\,\text{Re}\left(P_{\tilde{\mathcal{B}}}\right)\tilde{\mathbb{M}}^{-1}\,\text{Im}\left(P_{\tilde{\mathcal{A}}}^{t}\right)+\mathbb{I}_{g}\right)^{t}&-2\,\text{Re}\left(P_{\tilde{\mathcal{A}}}\right)\tilde{\mathbb{M}}^{-1}\,\text{Im}\left(P_{\tilde{\mathcal{A}}}^{t}\right)\\
2\,\text{Re}\left(P_{\tilde{\mathcal{B}}}\right)\tilde{\mathbb{M}}^{-1}\,\text{Im}\left(P_{\tilde{\mathcal{B}}}^{t}\right)&-\left(2\,\text{Re}\left(P_{\tilde{\mathcal{B}}}\right)\tilde{\mathbb{M}}^{-1}\,\text{Im}\left(P_{\tilde{\mathcal{A}}}^{t}\right)+\mathbb{I}_{g}\right)
\end{matrix}\right),\label{R5}
\end{equation} 
where 
\begin{equation}
    \tilde{\mathbb{M}}=\text{Im}\left(P_{\tilde{\mathcal{B}}}^{t}\right)\text{Re}\left(P_{\tilde{\mathcal{A}}}\right)-\text{Im}\left(P_{\tilde{\mathcal{A}}}^{t}\right)\text{Re}\left(P_{\tilde{\mathcal{B}}}\right).
    \label{Mtilde}
\end{equation}  \label{proposition R}
\end{proposition}

\begin{proof}
Using (\ref{hom basis}) we deduce the action of $\tau$ on (\ref{transf Vinn}):
\begin{equation}
\left(\begin{matrix}
A&B\\
C&D
\end{matrix}\right)^{-1}\left(\begin{matrix}
\mathbb{I}_{g}&\,\,0\\
\mathbb{H}\,\,&-\mathbb{I}_{g}
\end{matrix}\right)
\left(\begin{matrix}
\mathbf{\mathcal{A}}\\
\mathbf{\mathcal{B}}
\end{matrix}\right)=
\mathbf{R}\left(\begin{matrix}
A&B\\
C&D
\end{matrix}\right)^{-1}
\left(\begin{matrix}
\mathbf{\mathcal{A}}\\
\mathbf{\mathcal{B}}
\end{matrix}\right), \label{R2}
\end{equation}
which yields
\begin{equation}
\mathbf{R}
=
\left(\begin{matrix}
A&B\\
C&D
\end{matrix}\right)^{-1}\left(\begin{matrix}
\mathbb{I}_{g}&\,\,0\\
\mathbb{H}\,\,&-\mathbb{I}_{g}
\end{matrix}\right)
\left(\begin{matrix}
A&B\\
C&D
\end{matrix}\right). \label{R3}
\end{equation}
In other words, a symplectic matrix which transforms a basis 
$(\tilde{\mathcal{A}},\tilde{\mathcal{B}})$ to a symmetric form
 also satisfies (\ref{R3}). From (\ref{R3}) and (\ref{inverse})  one gets
\begin{equation}
\mathbf{R}=
\left(\begin{matrix}
(2\,C^{t}B-A^{t}\mathbb{H}B+\mathbb{I}_{g})^{t}&2\,D^{t}B-B^{t}\mathbb{H}B\\
-2\,C^{t}A+A^{t}\mathbb{H}A&-(2\,C^{t}B-A^{t}\mathbb{H}B+\mathbb{I}_{g})
\end{matrix}\right).\label{R4}
\end{equation}
Replacing $A$ via (\ref{mA}) and $B$ via (\ref{mB}) in (\ref{mC}) and 
then using (\ref{mA}) again to eliminate the factor $\left[C 
\,\text{Im}\left(P_{\tilde{\mathcal{A}}}\right)+D 
\,\text{Im}\left(P_{\tilde{\mathcal{B}}}\right)\right]$ leads to a 
relation for $C^{t}A$ only (using in the last step (\ref{mB}) instead 
of (\ref{mA}) gives a 
relation for $C^{t}B$). Similarly one gets a relation for $D^{t}B$, 
\begin{align}
2\,C^{t}A&=A^{t}\mathbb{H}A-2\,\text{Re}\left(P_{\tilde{\mathcal{B}}}\right)\tilde{\mathbb{M}}^{-1}\,\text{Im}\left(P_{\tilde{\mathcal{B}}}^{t}\right),\nonumber\\
2\,C^{t}B&=A^{t}\mathbb{H}B+2\,\text{Re}\left(P_{\tilde{\mathcal{B}}}\right)\tilde{\mathbb{M}}^{-1}\,\text{Im}\left(P_{\tilde{\mathcal{A}}}^{t}\right),\nonumber\\
2\,D^{t}B&=B^{t}\mathbb{H}B-2\,\text{Re}\left(P_{\tilde{\mathcal{A}}}\right)\tilde{\mathbb{M}}^{-1}\,\text{Im}\left(P_{\tilde{\mathcal{A}}}^{t}\right),\nonumber
\end{align}
with $\tilde{\mathbb{M}}$ given by (\ref{Mtilde}).
Substituting these relations in (\ref{R4}) one gets (\ref{R5}).
\end{proof}

\subsection{Symplectic transformation}
\label{matrix}

In this part we present in Theorem \ref{Theorem} the main result of 
the present paper: the symplectic transformation between an arbitrary 
basis $(\tilde{\mathcal{A}},\tilde{\mathcal{B}})$ on a real Riemann 
surface and a homology basis adapted to the symmetry is given in 
terms of the period matrices $P_{\tilde{\mathcal{A}}}, 
P_{\tilde{\mathcal{B}}}$ defined in the previous section. This result 
will allow to construct in section \ref{algo} an algorithm which computes the topological type of a given real Riemann surface.

We start with the following lemma which describes the spectral properties of the matrix $\mathbf{R}$ (\ref{R5}):

\begin{lemma}
The matrix $\mathbf{R}$ in (\ref{R5}) is diagonalizable over $\mathbb{Q}$ with eigenvalues $1$ and $-1$. The dimension of the corresponding eigenspaces equals $g$.
\label{Lemma vp}
\end{lemma}

\begin{proof}
It is straightforward to see that the matrix $\left(\begin{matrix}
\mathbb{I}_{g}&\,\,0\\
\mathbb{H}\,\,&-\mathbb{I}_{g}
\end{matrix}\right)$ is diagonalizable over $\mathbb{Q}$: the eigenvalues are $1$ and 
$-1$, and the dimension of the corresponding eigenspaces equals $g$. Therefore, by (\ref{R3}) the same holds for the matrix $\mathbf{R}$. 
\end{proof}

In what follows we denote by 
\begin{equation}
\mathcal{K}_{\Z}:=\left\{w\in\Z^{2g}; (\mathbf{R}^{t}-\mathbb{I}_{2g})w=0\right\} \label{K}
\end{equation} the integer kernel of the matrix $\mathbf{R}^{t}-\mathbb{I}_{2g}$. 
According to the theory of modules over the principal ring $\Z$, the 
$\Z$-module $\mathcal{K}_{\Z}$ admits a $\Z$-basis which can, for 
instance, be computed from the Smith normal form (see section \ref{module}).

\begin{theorem} Let $\Rs$ be a real compact Riemann surface of genus $g$
 and $(\mathbf{\tilde{\mathcal{A}}},\mathbf{\tilde{\mathcal{B}}})$ a 
 canonical
homology basis on $\Rs$. For the given matrices of periods $P_{\tilde{\mathcal{A}}}, P_{\tilde{\mathcal{B}}}$ of holomorphic differentials satisfying (\ref{diff mu}), let  $\left(\begin{matrix}
S_{1}\\S_{2}
\end{matrix}\right)$ be the $2g\times g$ matrix formed by a 
$\Z$-basis of $\mathcal{K}_{\Z}$ (\ref{K}). Then a symplectic matrix in (\ref{transf Vinn}) which relates the homology basis $(\mathbf{\tilde{\mathcal{A}}},\mathbf{\tilde{\mathcal{B}}})$ to a symmetric homology basis $(\mathbf{\mathcal{A}},\mathbf{\mathcal{B}})$ is given by:
\begin{align}\label{AB}
\left(\begin{matrix}
A^{t}\\B^{t}
\end{matrix}\right)&=\left(\begin{matrix}
S_{1}\\S_{2}
\end{matrix}\right)Q\\
\left(\begin{matrix}
C^{t}\\D^{t}
\end{matrix}\right)&=\frac{1}{2}\left(\begin{matrix}
S_{1}\\
S_{2}
\end{matrix}\right)Q\mathbb{H}+
\left(\begin{matrix}
-\text{Re}\left(P_{\tilde{\mathcal{B}}}\right)\\
\,\,\,\,\text{Re}\left(P_{\tilde{\mathcal{A}}}\right)
\end{matrix}\right)\left[S_{1}^{t} \,\text{Re}\left(P_{\tilde{\mathcal{A}}}\right)+S_{2}^{t}\, \text{Re}\left(P_{\tilde{\mathcal{B}}}\right)\right]^{-1}(Q^{t})^{-1} \label{CD}
\end{align} 
where the matrix $\mathbb{H}$ is defined in section 2, and where $Q\in GL_{g}(\Z)$ is such that
\begin{equation}
\left(\begin{matrix}
S_{1}\\
S_{2}
\end{matrix}\right)Q\mathbb{H}Q^{t}
 \equiv  2\left(\begin{matrix}
-\text{Re}\left(P_{\tilde{\mathcal{B}}}\right)\\
\,\,\,\,\text{Re}\left(P_{\tilde{\mathcal{A}}}\right)
\end{matrix}\right)\left[S_{1}^{t} \,\text{Re}\left(P_{\tilde{\mathcal{A}}}\right)+S_{2}^{t}\, \text{Re}\left(P_{\tilde{\mathcal{B}}}\right)\right]^{-1}  \quad (\text{mod} \,\, 2). \label{cond Q}
\end{equation}  \label{Theorem}
\end{theorem}

\begin{proof}

Notice that (\ref{R3}) can be rewritten as
\begin{equation}
\left(\begin{matrix}
A&B\\
C&D
\end{matrix}\right)\mathbf{R}
=
\left(\begin{matrix}
\mathbb{I}_{g}&\,\,0\\
\mathbb{H}\,\,&-\mathbb{I}_{g}
\end{matrix}\right)
\left(\begin{matrix}
A&B\\
C&D
\end{matrix}\right), \label{R3bis}
\end{equation}
which, in particular, gives the following condition for the matrices $A$ and $B$:
\begin{equation}
\left(\begin{matrix}
A&B
\end{matrix}\right)
\left(\mathbf{R}-\mathbb{I}_{2g}\right)=0. \label{AB ker}  
\end{equation}
 Denote by $u_{1},\ldots,u_{g},v_{1},\ldots,v_{g}$ the column vectors  of the matrix $\left(\begin{matrix}
A&B\\
C&D
\end{matrix}\right)^{t}$. By (\ref{AB ker}) the vectors $u_{i}$ for 
$i=1,\ldots,g$ lie in the integer kernel $\mathcal{K}_{\Z}$ of the matrix $\mathbf{R}^{t}-\mathbb{I}_{2g}$.  
Let us prove that $(u_{1},\ldots,u_{g})$ form in fact a $\Z$-basis of the module $\mathcal{K}_{\Z}$. 
By Lemma \ref{Lemma vp} one has $dim(\mathcal{K}_{\mathbb{Q}})=g$, where $\mathcal{K}_{\mathbb{Q}}$ denotes the kernel of $\mathbf{R}^{t}-\mathbb{I}_{2g}$ over the field $\mathbb{Q}$,   which by Theorem \ref{theo1} yields  $rank(\mathcal{K}_{\Z})= g$. Therefore one has to check that $(u_{1},\ldots,u_{g})$ are free vectors over $\Z$ and generate the $\Z$-module $\mathcal{K}_{\Z}$. Notice that here it is important to check that these vectors form a set of generators for the $\Z$-module $\mathcal{K}_{\Z}$, as we saw in Remark \ref{remark module}.

Since $\left(\begin{matrix}
A&B\\
C&D
\end{matrix}\right)^{t}\in GL_{2g}(\Z)$, the $2g$ vectors 
$u_{1},\ldots,u_{g},v_{1},\ldots,v_{g}$ form a $\Z$-basis of the 
module $\Z^{2g}$, which in particular implies that these vectors 
  are free over $\Z$. 
  Then it remains to prove that the vectors $u_{i}, i=1,\ldots,g$ 
generate the $\Z$-module  $\mathcal{K}_{\Z}$ which is done by 
contradiction as follows. Let $w\in \mathcal{K}_{\Z}$ which we write 
in $\Z^{2g}$ as 
$w=\sum_{i=1}^{g}\alpha_{i}u_{i}+\sum_{j=1}^{g}\beta_{j}v_{j}$ with 
$\alpha_{i},\beta_{j}\in\Z$ such that at least one of the $\beta_{j}$ is non-zero. 
Since $w,u_{1},\ldots,u_{g}\in\mathcal{K}_{\Z}$, one has $v:=\sum_{j=1}^{g}\beta_{j}v_{j}\in\mathcal{K}_{\Z}$ and $v\neq 0$.
We deduce that $u_{1},\ldots,u_{g},v$ are $g+1$ free vectors in 
$\mathcal{K}_{\Z}$. This is impossible since  
$rank(\mathcal{K}_{\Z})= g$. Thus one has $\beta_{j}=0$ for $j=1,\ldots,g$ which implies that the vectors $u_{i}, i=1,\ldots,g$ generate the $\Z$-module  $\mathcal{K}_{\Z}$.

%
%First let us check that $dim(\mathcal{K}_{\Z})\leq g$. It is clear that a $\Z$-basis of the module $\mathcal{K}_{\Z}$ is also a basis of the $\mathbb{Q}$-vector space $\mathcal{K}_{\mathbb{Q}}$, then   $dim(\mathcal{K}_{\Z})\leq dim(\mathcal{K}_{\mathbb{Q}})$. By Lemma \ref{Lemma vp} one has $dim(\mathcal{K}_{\mathbb{Q}})=g$ which therefore yields  $dim(\mathcal{K}_{\Z})\leq g$. Moreover, since $\left(\begin{matrix}
%A&B\\
%C&D
%\end{matrix}\right)\in Gl_{2g}(\Z)$, the $g$ column vectors of the matrix $\left(\begin{matrix}
%A&B
%\end{matrix}\right)^{t}$ are free over $\Z$ and then form a $\Z$-basis of the module $\mathcal{K}_{\Z}$.

Hence we can write
\begin{equation}
\left(\begin{matrix}
A^{t}\\B^{t}
\end{matrix}\right)
=
\left(\begin{matrix}
S_{1}\\S_{2}
\end{matrix}\right)Q,  \label{S1 S2}
\end{equation}
for some $Q\in GL_{g}(\Z)$, where the $g$ column vectors of the matrix $\left(\begin{matrix}
S_{1}\\S_{2}
\end{matrix}\right)$ form a $\Z$-basis of the module $\mathcal{K}_{\Z}$. Here the matrix $Q$ encodes  the freedom in the choice of such a basis. 
The matrices $C$ and $D$ are then given by (\ref{mC}) and (\ref{mD}). It follows that these two matrices are integer matrices if and only if the matrix $Q$ satisfies (\ref{cond Q}), which completes the proof.
\end{proof}

\begin{remark}
\emph{If the curve is an M-curve, then $\mathbb{H}=0$ and the matrices $A,B,C,D$ in Theorem \ref{Theorem} are given by
\begin{align}
\left(\begin{matrix}
A^{t}\\B^{t}
\end{matrix}\right)&=\left(\begin{matrix}
S_{1}\\S_{2}
\end{matrix}\right)Q\\
\left(\begin{matrix}
C^{t}\\D^{t}
\end{matrix}\right)&=
\left(\begin{matrix}
-\text{Re}\left(P_{\tilde{\mathcal{B}}}\right)\\
\,\,\,\,\text{Re}\left(P_{\tilde{\mathcal{A}}}\right)
\end{matrix}\right)\left[S_{1}^{t} \,\text{Re}\left(P_{\tilde{\mathcal{A}}}\right)+S_{2}^{t}\, \text{Re}\left(P_{\tilde{\mathcal{B}}}\right)\right]^{-1}(Q^{t})^{-1}
\end{align} 
where $Q\in GL_{g}(\Z)$ is arbitrary.}

\end{remark}

\begin{remark}
\emph{As explained below, if the curve is not an M-curve, namely, $\mathbb{H}\neq 0$, one can construct explicitly a matrix 
$Q\in GL_{g}(\Z/2\Z)$ such that (\ref{cond Q}) holds. }
\end{remark}

This construction of $Q$ is based on the Smith normal form of the matrix $\left(\begin{matrix}
S_{1}\\
S_{2}
\end{matrix}\right)$ (see section \ref{module}) which allows the 
simplification of the system (\ref{cond Q}).

\begin{lemma}
There exist $U\in GL_{2g}(\Z), \,V\in GL_{g}(\Z)$   such that
\begin{equation}
U\left(\begin{matrix}
S_{1}\\
S_{2}
\end{matrix}\right)V=
\left(\begin{matrix}
\mathcal{E}\\
0
\end{matrix}\right),  \label{Smith S}
\end{equation}
where $\mathcal{E}$ is a $g\times g$ diagonal matrix with elements $\mathcal{E}_{ii}=\pm1$ 
for $i=1,\ldots,g$.  \label{U}
\end{lemma}

\begin{proof}
By (\ref{Smith}) there exist $U\in GL_{2g}(\Z), \,V\in 
GL_{g}(\Z)$  and $p_{1},\ldots,p_{g}\in\N\setminus\{0\}$ satisfying   $p_{1}|p_{2}|\ldots|p_{g}$ such that
\begin{equation}
U\left(\begin{matrix}
S_{1}\\
S_{2}
\end{matrix}\right)V=
\left(\begin{matrix}
\mathbb{D}\\
0
\end{matrix}\right),  \label{Smith Sbis}
\end{equation}
where $\mathbb{D}:=\text{Diag}(p_{1},\ldots,p_{g})$. The fact that 
$\mathbb{D}=\mathcal{E}$ can be deduced from the following equalities:
\begin{align}
1&=\det\left(\begin{matrix}
A&B\\
C&D
\end{matrix}\right)
=\det\left(\begin{matrix}
S_{1}&C^{t}\\
S_{2}&D^{t}
\end{matrix}\right)\det(Q)
=\pm \det\left(\begin{matrix}
\mathbb{D}&F_{1}\\
0&F_{2}
\end{matrix}\right)=\pm \det(\mathbb{D})\det(F_{2})\label{det1}
\end{align}
where we  multiplied the matrix in the 
determinant from the left by $U$ and from the right by $
\begin{pmatrix}
    V & 0 \\
    0 & \mathbb{I}_{g}
\end{pmatrix}
$, and used $\det(Q)\det(U)\det(V)=\pm 1$ since 
$Q,V\in GL_{g}(\mathbb{Z})$ and $U\in GL_{2g}(\mathbb{Z})$; here $
\begin{pmatrix}
    F_{1} \\
    F_{2}
\end{pmatrix}=U
\begin{pmatrix}
    C^{t} \\
    D^{t}
\end{pmatrix}
$.
We deduce that $\det(\mathbb{D})=\pm 1$ since the right-hand side of 
(\ref{det1}) is a product of determinants of 
matrices with integer coefficients. This completes the proof.
\end{proof}

Now let us define matrices $N_{1},N_{2}\in \mathcal{M}_{g}(\Z)$ as follows:
\begin{equation}
\left(\begin{matrix}
N_{1}\\
N_{2}
\end{matrix}\right):=2\,VU\left(\begin{matrix}
-\text{Re}\left(P_{\tilde{\mathcal{B}}}\right)\\
\,\,\,\,\text{Re}\left(P_{\tilde{\mathcal{A}}}\right)
\end{matrix}\right)\left[S_{1}^{t} \,\text{Re}\left(P_{\tilde{\mathcal{A}}}\right)+S_{2}^{t}\, \text{Re}\left(P_{\tilde{\mathcal{B}}}\right)\right]^{-1}.  \label{N1}
\end{equation}
Then one has:

\begin{proposition}
$Q\in GL_{g}(\Z)$ satisfies (\ref{cond Q}) if and only if it solves 
\begin{equation}
 Q\mathbb{H}Q^{t}\equiv  N_{1} \quad(\text{mod} \,\, 2), \label{Q}
\end{equation}
where $N_{1}\in \mathcal{M}_{g}(\Z)$ is defined in (\ref{N1}). 
Moreover, if $\tilde{Q}$ is a particular solution of (\ref{Q}) then the 
general solution can be written as $\tilde{Q}Q_{0}$ where $Q_{0}$ solves 
\begin{equation}
 Q_{0}\mathbb{H}Q_{0}^{t}\equiv  \mathbb{H} \quad(\text{mod} \,\, 2). \label{Q0}
\end{equation}

\end{proposition}

\begin{proof}
Multiplying  equality (\ref{cond Q}) from the left by the matrix $U$ 
of Lemma \ref{U} and using (\ref{Smith S}) one gets 
\begin{equation}
\left(\begin{matrix}
\mathbb{I}_{g}\\
0
\end{matrix}\right)V^{-1}Q\mathbb{H}Q^{t}
 \equiv  2\,U\left(\begin{matrix}
-\text{Re}\left(P_{\tilde{\mathcal{B}}}\right)\\
\,\,\,\,\text{Re}\left(P_{\tilde{\mathcal{A}}}\right)
\end{matrix}\right)\left[S_{1}^{t} \,\text{Re}\left(P_{\tilde{\mathcal{A}}}\right)+S_{2}^{t}\, \text{Re}\left(P_{\tilde{\mathcal{B}}}\right)\right]^{-1}  \quad (\text{mod} \,\, 2), \nonumber
\end{equation}
which is equivalent to
\begin{equation}
\left(\begin{matrix}
\mathbb{I}_{g}\\
0
\end{matrix}\right)Q\mathbb{H}Q^{t}
 \equiv  2\,VU\left(\begin{matrix}
-\text{Re}\left(P_{\tilde{\mathcal{B}}}\right)\\
\,\,\,\,\text{Re}\left(P_{\tilde{\mathcal{A}}}\right)
\end{matrix}\right)\left[S_{1}^{t} \,\text{Re}\left(P_{\tilde{\mathcal{A}}}\right)+S_{2}^{t}\, \text{Re}\left(P_{\tilde{\mathcal{B}}}\right)\right]^{-1}  \quad (\text{mod} \,\, 2). \nonumber
\end{equation}
Using the definition (\ref{N1}) one gets (\ref{Q}). 

Now to check that a transformation of the form $Q\rightarrow QQ_{0}$ 
where $Q_{0}$ solves (\ref{Q0}) is the only one which preserves 
(\ref{Q}), notice that such a transformation corresponds to a 
symplectic transformation between the  symmetric homology basis 
obtained from Theorem \ref{Theorem} and another symmetric homology basis, since this coincides with a change of the $\Z$-basis in (\ref{AB}). 
Hence from (\ref{hom basis}) it is straightforward to see that the 
symplectic matrix which relates two symmetric homology bases  is given by
\begin{equation}
\left(\begin{matrix}
Q_{0}^{t} & 0\\
\frac{1}{2}(\mathbb{H}Q_{0}^{t}-Q_{0}^{-1}\mathbb{H})& Q_{0}^{-1}   \label{symp sym}
\end{matrix}\right),
\end{equation}
with $Q_{0}\in GL_{g}(\Z)$ satisfying $Q_{0}\mathbb{H}Q_{0}^{t}\equiv \mathbb{H} \,\,(\text{mod} \,\, 2)$.
This completes the proof.
\end{proof}

% \begin{remark}
% From Theorem \ref{Theorem} one can easily deduced the symplectic transformation between two arbitrary homology basis on the same real Riemann surface.
% \end{remark}

\section{Algorithm for the computation of the topological type $(g,k,a)$}

\label{algo}

The results of the previous section allow us to formulate an 
algorithm to transform an arbitrary canonical homology basis 
$(\tilde{\mathcal{A}},\tilde{\mathcal{B}})$, for 
instance obtained via the algorithm \cite{tretalg}, to a symmetric 
basis $(\mathcal{A},\mathcal{B})$ satisfying (\ref{hom basis}). The 
key task in this context is the computation of the matrix $Q$ in 
(\ref{Q}). In the process of computing $Q$, the matrix $\mathbb{H}$ 
giving the topology of the real Riemann surface can be determined.

The starting point of the algorithm are the periods $P_{\tilde{\mathcal{A}}}$ 
and $P_{\tilde{\mathcal{B}}}$ of a basis of differentials 
$(\nu_{1},\ldots,\nu_{g})$ satisfying (\ref{diff mu}). Notice that 
condition (\ref{diff mu}) is important. The Maple 
\emph{algcurves} package generates such differentials for real curves 
by default. For the Matlab code used in this paper this is in general 
not the case if a numerically optimal approach is used to determine 
the holomorphic differentials, see \cite{FK} for 
details. However, it is possible to determine a basis of the 
holomorphic differentials with rational coefficients which will satisfy condition 
(\ref{diff mu}).
For the examples in the following section, we always choose 
this option.

With these periods the code computes the matrix $\mathbf{R}$ via 
(\ref{R5}). This matrix will have integer entries up to the used 
precision (by default $10^{-6}$ in Maple and $10^{-12}$ in Matlab). 
Rounding has to be used to obtain an integer matrix. 
Computing the Smith normal 
form (\ref{Smith}) of the matrix $\mathbf{R}^{t}-\mathbb{I}_{2g}$, we 
obtain from the last $g$ vectors of the resulting matrix $V$ a $\Z$-basis of the integer kernel  
$\mathcal{K}_{\Z}$ (\ref{K}), 
i.e., the column vectors of the matrix $
\begin{pmatrix}
    S_{1} \\
    S_{2}
\end{pmatrix}
$ in (\ref{S1 S2}). 

An algorithm to compute the Smith normal form  $UMV=S$ of an integer matrix 
$M$ is 
implemented in Maple. This algorithm can be called from Matlab via 
the symbolic toolbox. We use here an own implementation of the standard 
algorithm to compute the Smith normal form which we briefly summarize: we always work on column $j$ 
starting with $j=1$. If there is no nonzero element in this column, 
it is swapped by multiplication with an appropriate matrix $U$ from 
the left with the last column with a nonzero element. If the element 
$M_{jj}=0$, a row with a nonzero element in position $j$ is added (to 
avoid clumsy notation, the transformed matrix is still called $M$). 
Then all nonzero elements $M_{jk}$ for $k\neq j$ are eliminated by 
adding row $j$ with appropriate multipliers obtained via the 
Euclidean algorithm. In the same way the row with index $j$ is 
cleaned by acting on $M$ via multiplication by a matrix $V$ from the 
right. If the resulting element $M_{jj}$ does not divide all other 
elements of $M$, one of these elements is added by using the 
Euclidean algorithm to $M_{jj}$ in a way that the latter becomes 
smaller. This destroys possibly the nullity of the remaining elements 
in column $j$ and row $j$ which thus have to be cleaned as before. 
This process is repeated until $M_{jj}$ divides all other elements of 
$M$. Then the index $j$ is incremented by 1. The procedure is 
repeated until the Smith normal form is obtained. 

Notice that this standard algorithm has a well known problem: in 
general for larger matrices the entries of the matrices $U$ and $V$ 
become very large. Though these are integer matrices, this is 
problematic once some of the entries are of the order of $10^{16}$ (machine 
precision in Matlab is $10^{-16}$ which implies that integers of the 
order of $10^{16}$, which are internally treated as floating point 
numbers, can no longer be numerically distinguished). There are more 
sophisticated algorithms to treat larger matrices as the one given in \cite{HafCur}. 
In practice the standard algorithm works well for examples of a genus 
$g\leq 6$ which is sufficient for our purposes. Only for higher 
genus, the algorithm \cite{HafCur} would be needed. 

By computing the Smith normal form of the vector $
\begin{pmatrix}
    S_{1} \\
    S_{2}
\end{pmatrix}$
in (\ref{Smith S}), we get the needed quantities to compute the matrix $N_{1}$ 
in (\ref{N1}).
The main task is then to determine the matrix $Q$ in (\ref{cond Q}) since for 
given periods, the whole symplectic matrix (\ref{transf Vinn}) 
follows from equations (\ref{AB}) and (\ref{CD}) for given $
\begin{pmatrix}
    S_{1} \\
    S_{2}
\end{pmatrix}
$. The matrices $Q$ and $\mathbb{H}$ can be determined from relation 
(\ref{Q}) for a given matrix $N_{1}$ by 
standard Gaussian elimination in  $\mathbb{Z}/2\mathbb{Z}$ and by 
imposing  the block diagonal form of section 2 on $\mathbb{H}$ as we 
will outline below.

\begin{remark}\label{rem}
\emph{It was shown in section 2 that the matrix $\mathbb{H}$ can be 
chosen to be either diagonal or to consist of blocks of the form 
$$   H_{0}= 
    \begin{pmatrix}
        0 & 1 \\
        1 & 0
    \end{pmatrix}.
$$ A unique determination of the matrix $\mathbb{H}$ in the  computation of $Q$ from 
(\ref{Q}) is only possible, if this block $H_{0}$  cannot be related 
through a similarity transformation in  $\mathbb{Z}/2\mathbb{Z}$ to 
$\mathbb{I}_{2}$. In 
fact this is the case since $H_{0}$ cannot be diagonalized in  
$\mathbb{Z}/2\mathbb{Z}$.
The same reasoning applies if the 
matrix $\mathbb{H}$ consists of several blocks $H_{0}$ and 
zeros otherwise.\\
However, a block of the form
$$
H_{1}=\begin{pmatrix}
    1 & 0 & 0 \\
    0 & 0 & 1 \\
    0 & 1 & 0
\end{pmatrix}
$$ can be diagonalized in  $\mathbb{Z}/2\mathbb{Z}$ by multiplication 
from the left and the right by a matrix of the form
$$
Q_{1}=\begin{pmatrix}
    1 & 1 & 1 \\
    1 & 0 & 1 \\
    1 & 1 & 0
\end{pmatrix}.
$$ 
It follows 
that if the matrix  $\mathbb{H}$ determined in 
the computation of $Q$ from (\ref{Q}) has a non-zero diagonal element, 
then $\mathbb{H}$ can be diagonalized in several steps: if the matrix 
$\mathbb{H}$ has the form $\tilde{H}$ below 
($\tilde{H}_{ij}=\delta_{ij}$ for $i,j<k$ and a block $H_{0}$ for 
$i,j=k,k+1$)
$$
\tilde{H}=\begin{pmatrix}
    1 &  &  &  &  &  &  &  \\
     & \ddots &  &  &  &  &  &  \\
     &  & 1 &  &  &  &  &  \\
     &  &  & 0 & 1 &  &  &  \\
     &  &  & 1 & 0 &  &  &  \\
     &  &  &  &  &  &  &  \\
     &  &  &  &  &  &  &  \\
     &  &  &  &  &  &  & 
\end{pmatrix}
\qquad
\tilde{Q}=\begin{pmatrix}
    1 &  &  & 1  & 1 &  &  &  \\
     &  &  &  &  &  &  &  \\
     &  &  &  &  &  &  &  \\
    1 &  &  & 0 & 1 &  &  &  \\
    1 &  &  & 1 & 0 &  &  &  \\
     &  &  &  &  &  &  &  \\
     &  &  &  &  &  &  &  \\
     &  &  &  &  &  &  & 
\end{pmatrix}
$$
then multiplication from the left and the right with a matrix 
$\tilde{Q}$ (not shown elements of this matrix are 0) gives the identity matrix for $i,j\leq k+1$. Applying 
this procedure several times will lead to a diagonal 
matrix $\mathbb{H}$. }
\end{remark}

The algorithm for the computation of $Q$ and $\mathbb{H}$ via the similarity relation 
(\ref{Q}) for given $N_{1}$ by imposing a block diagonal form for the 
matrix $\mathbb{H}$ as in section 2 uses  
in principle standard Gauss elimination on the rows and 
columns of $N_{1}$ over the field $\mathbb{Z}/2\mathbb{Z}$ with minor 
modifications as detailed below. We only describe 
the action on the columns via a matrix $q$ from the right, since the action on the rows follows by 
symmetry by multiplication with $q^{t}$ from the left: \\
- if $N_{1}\equiv 0 \,(\text{mod}\, 2)$, put $\mathbb{H}=0$ and $Q=\mathbb{I}_{g}$ and end the algorithm, otherwise, put the index $j$ of the column under consideration equal to 1;\\
- if column $j$ contains only zeros, it   
is swapped with the last column with non-zero entries;\\
- if there is a 1 in position $j$ of the column, all further 
non-zero entries in the 
column are eliminated in standard way;\\
- if there is a 1 in the column, but not in position $j$,  
rows are swapped in a way that it appears in the position $j+1$ of 
the column (it cannot be put to position $j$ as explained in 
Remark~\ref{rem}). Further ones in the column are eliminated; \\
- if there was  a non-zero diagonal entry, the column index $j$ is 
incremented by 1, if there was a block $ \begin{pmatrix}
        0 & 1 \\
        1 & 0
    \end{pmatrix},$
the index $j$ is incremented by 2. Then the 
algorithm is repeated with column $j$ until $j=g$ or 
until the columns with index $j$ and higher only contain zeros; \\
- if there are blocks of the form $\tilde{H}$ in Remark~\ref{rem}, then 
$\mathbb{H}$ will be diagonalized by multiplication with the 
corresponding matrix 
$\tilde{Q}$ as explained in \ref{rem}.

\section{Examples}
In this section we study examples of real algebraic curves, 
provide the computed periods and the application of the algorithm to 
obtain a symmetric homology basis as well as the matrix $\mathbb{H}$ 
encoding the topological information of the curve. For convenience we 
use here the Matlab algebraic curves package, but the same examples 
can be of course studied with the Maple package. We also give 
graphical representations of the real variety of an algebraic 
curve if there is any 
%as in %Fig.~\ref{figtrott} 
which is generated via contour plots of 
$f(x,y)=0$ for real $x $ and $y$ (this corresponds to the command 
{\tt plot\_real\_curve} in the Maple \emph{algcurves} package). This 
is not identical with the set $\mathcal{R}(\R)$ of real ovals of the 
Riemann surface since the curves may have singularities, whereas the 
Riemann surface is defined by desingularized such curves (see 
\cite{FK} for how this is done in the Matlab package). Thus there may 
be cusps and self-intersections in the shown plots. Moreover these plots are not conclusive if 
the curves come very close, and if there are self-intersections as in 
Fig.~\ref{figx9} below. In addition we only show the curves for finite 
values of $x$ and $y$ from which it cannot be decided which curves 
cross at infinity and which lines belong to the same ovals as in 
Fig.~\ref{figdivid}. They just serve for illustration purpose, for 
more sophisticated approaches, see \cite{Arnon, Coste, Sak, Feng, Gonz,  Hong, 
Seidel}. The computed number of real ovals via the 
algorithm is, however, unique: as outlined in section 2, it follows 
from the rank of the matrix $\mathbb{H}$ (for $k\neq0$ one has 
$k=g+1-\mbox{rank}(\mathbb{H})$). We always assume in the following that 
it is known whether there are any real ovals. This allows  the 
unique identification of the topological type $(g,k,a)$ via the matrix $\mathbb{H}$.

The Trott curve \cite{Trott} given by the algebraic equation 
\begin{equation}
    144\,(x^{4}+y^{4})-225\,(x^{2}+y^{2})+350\,x^{2}y^{2}+81=0
    \label{trott}
\end{equation}
is known to be an M-curve  of genus 3 (it has the maximal number $g+1=4$ of real ovals, as can be seen in Fig.~\ref{figtrott}). Moreover, this curve has real 
branch points only (and 28 real bitangents, namely, tangent lines to the curve in two 
places). 
\begin{figure}[htb!]
\begin{center}
\includegraphics[width=0.7\textwidth]{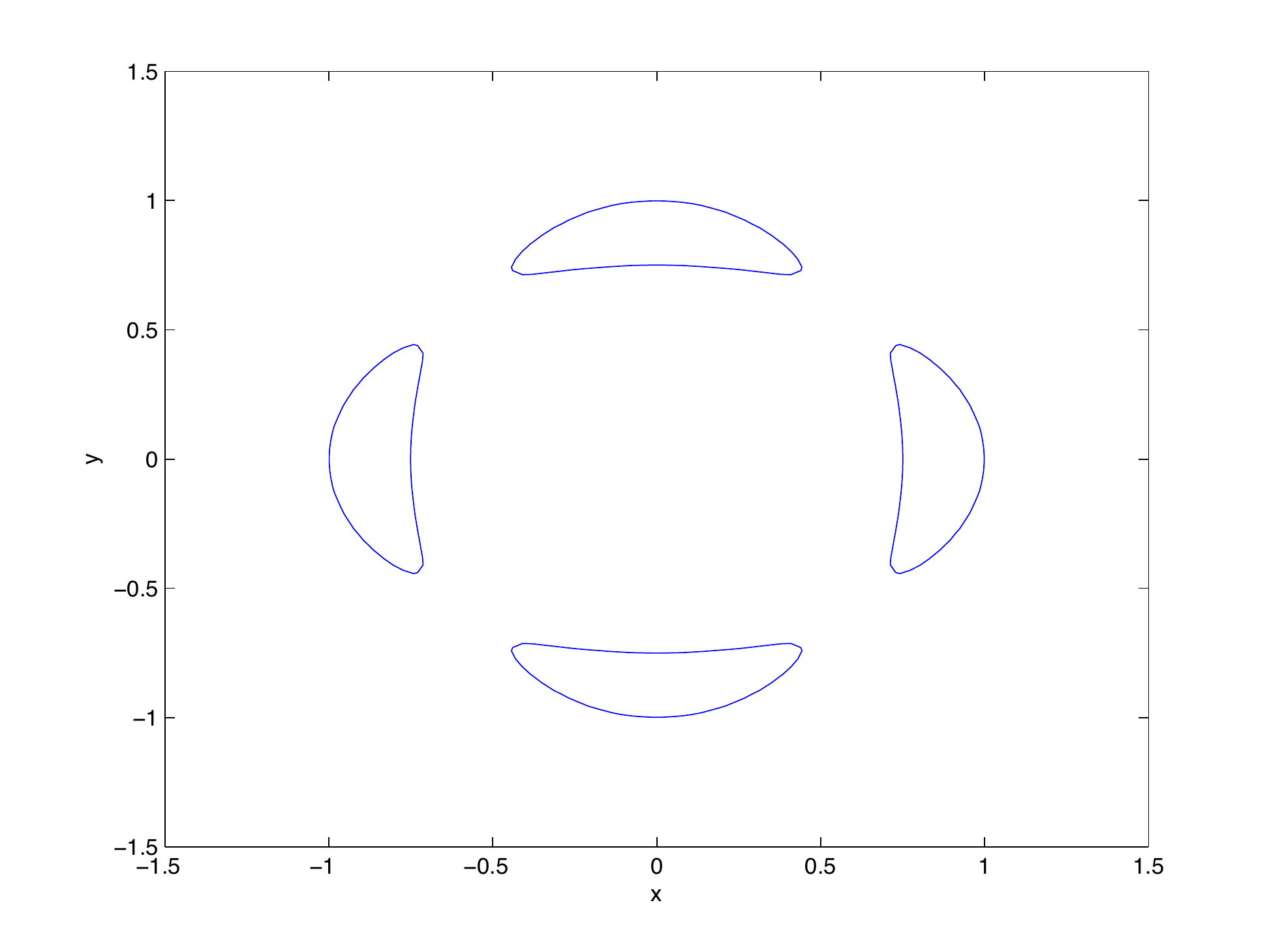}
\end{center}
 \caption{Real ovals of the Trott curve (\ref{trott}), an M-curve of 
 genus 3.}
 \label{figtrott}
\end{figure}
Our computed matrices of $\tilde{\mathcal{A}}$ and 
$\tilde{\mathcal{B}}$-periods denoted by    
\texttt{aper} and \texttt{bper} respectively read\footnote{For the ease of 
representation we only give 4 digits here, though at least 12 digits 
are known for these quantities.}
\begin{verbatim}
aper =

  -0.0000 + 0.0235i  -0.0000 + 0.0138i  -0.0000 + 0.0138i
        0 + 0.0000i   0.0000 + 0.0277i        0 + 0.0000i
  -0.0315            -0.0000 + 0.0000i   0.0250 - 0.0000i
  
bper =

  -0.0315 + 0.0235i  -0.0000 + 0.0138i  -0.0250 + 0.0138i
  -0.0000 + 0.0000i  -0.0250 + 0.0277i   0.0250 - 0.0000i
  -0.0000 - 0.0235i   0.0000 + 0.0138i        0 + 0.0138i.  
\end{verbatim}
For this the algorithm produces as expected $\mathbb{H}=0$ and $Q$ of 
(\ref{Q}) the 
identity matrix. 
The symplectic transformation found by the  algorithm via (\ref{AB}) 
and (\ref{CD}) has the form
\begin{verbatim}
    [A,B,C,D] =

     0     1     0     0    -1     0     0     1     0     0     0     0
     1     0     0    -1     0     0     1     0     0     0     0     0
     0     0     1     0     0     0     0     0     0     0     0     1.
\end{verbatim}
Since we will not actually use the matrices $A,B,C,D$ in this 
article, and since they follow for given periods from $\mathbb{H}$ 
and $Q$ via 
(\ref{AB}) and (\ref{CD}) we will only give them for this example. 
However they are 
needed e.g.~for the study of algebro-geometric solutions to integrable 
equations as  KP and DS  as in \cite{KKnum}.

% The hyperelliptic curve 
% \begin{equation}
%     y^2 - x(x^6-1)=0
%     \label{hyper}
% \end{equation}
% of genus 3 has 2 real ovals as can be seen in Fig.~\ref{fighyper}.
% \begin{figure}[htb!]
% \begin{center}
% \includegraphics[width=0.6\textwidth]{realhyper.pdf}
% \end{center}
%  \caption{Real ovals of the hyperelliptic curve (\ref{hyper}).}
%  \label{fighyper}
% \end{figure}
% This is in accordance with the reslts of the algorithm. For the 
% periods 
% \begin{verbatim}
% aper =
% 
%    0.2977 + 0.5157i   1.7480 + 0.0000i   1.1111 - 1.9244i
%   -1.9244 + 0.5157i  -0.0000 + 0.0000i   0.5157 - 1.9244i
%         0 - 0.5954i  -0.0000 + 0.8740i        0 - 2.2222i
% 
% bper =
% 
%   -0.8134 + 0.8134i   0.8740 + 0.8740i   0.8134 - 0.8134i
%   -0.8134 - 1.4088i   0.8740 + 0.0000i   0.8134 - 1.4088i
%    0.2977 + 0.5157i  -0.8740 + 0.0000i   1.1111 - 1.9244i,    
% \end{verbatim}
% it finds 
% \begin{verbatim}
%     H =
% 
%      1     0     0
%      0     1     0
%      0     0     0
% 
% 
% Q =
% 
%      0     1     0
%      1     1     0
%      0     0     1,
% \end{verbatim}
% which, according to the definition of $\mathbb{H}$ (see section 2), gives the topological type $(3,2,1)$ for the curve. In other words, this  curve has genus $3$, two real ovals, and is non-dividing. 

The Klein curve given by the equation 
\begin{equation}
    y^7-x(x-1)^2=0
    \label{klein}
\end{equation}
has the maximal number of automorphisms  (168) of a genus 3 curve. The 
computed periods read
\begin{verbatim}
aper =

  -0.9667 + 0.7709i   0.9667 + 0.2206i   0.9667 - 2.0073i
  -1.2054 - 0.2751i  -0.4302 + 0.8933i  -1.7419 + 1.3891i
  -0.4302 - 0.8933i   1.7419 + 1.3891i  -1.2054 + 0.2751i

bper =

  -2.7085 - 0.6182i  -0.2387 + 0.4958i   1.3969 - 1.1140i
  -2.1721 - 1.7322i   0.5365 - 0.1224i  -0.7752 - 1.6097i
   0.9667 + 0.2206i  -0.9667 + 2.0073i  -0.9667 + 0.7709i.
\end{verbatim}
The algorithm finds 
\begin{verbatim}
H =

     1     0     0
     0     1     0
     0     0     1
Q =

     1     1     1
     0     0     1
     0     1     0.
\end{verbatim}
Therefore, the topological type of the curve is $(3,1,1)$, namely, the curve has genus $3$, one real oval (as can be also seen in Fig.~\ref{figklein}) and is non-dividing.  
\begin{figure}[htb!]
\begin{center}
\includegraphics[width=0.6\textwidth]{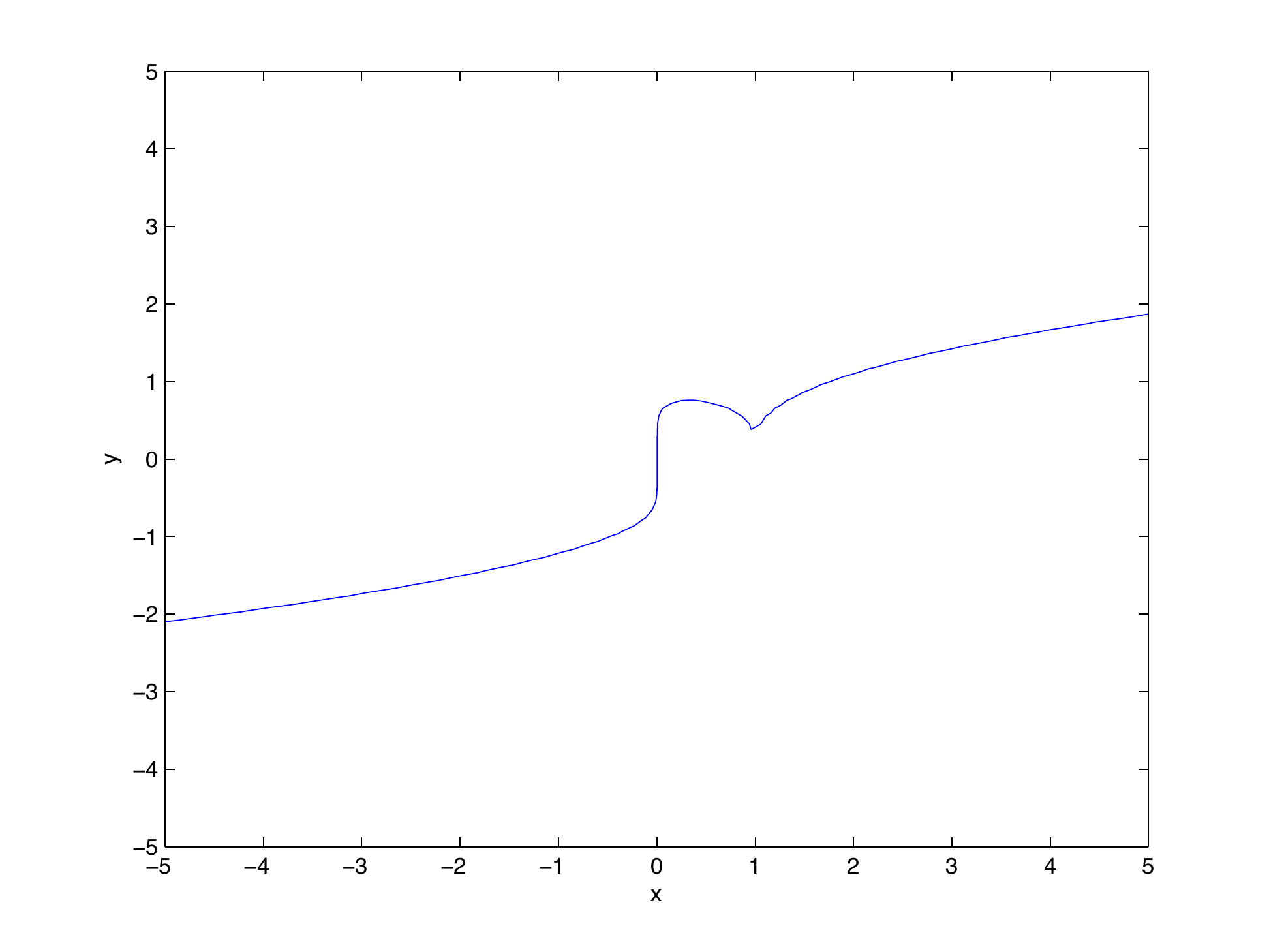}
\end{center}
 \caption{Real variety of the Klein curve (\ref{klein}), the curve 
 with the maximal number  of automorphisms in  genus 3.}
 \label{figklein}
\end{figure}

The Fermat curve 
\begin{equation}
    y^{n}+x^{n}+1=0
    \label{fermat}
\end{equation}
has for $n=4$ the topological type $(3,0,1)$. This is confirmed by the 
algorithm. For the periods 
\begin{verbatim}
aper =

   0.9270 + 0.0000i   0.0000 - 0.9270i   0.0000 - 0.9270i
  -0.0000 + 0.0000i   0.0000 + 0.0000i   0.0000 - 1.8541i
        0 + 0.9270i  -0.9270 + 0.0000i   0.0000 - 0.9270i

bper =

   0.9270 + 0.9270i   0.9270 - 0.9270i   0.0000 + 0.0000i
   0.0000            -0.9270 + 0.9270i   0.9270 - 0.9270i
  -0.9270 + 0.0000i   0.0000 - 0.9270i   0.0000 - 0.9270i
    
\end{verbatim}
we find
\begin{verbatim}
H =

     0     1     0
     1     0     0
     0     0     0


Q =

     1     0     0
     0     0     1
     0     1     0
\end{verbatim}
in accordance with the expectation.
For $n=5$ the curve has genus 6. We find with 
\begin{verbatim}
aper =

  Columns 1 through 4

  -0.1623 + 0.4995i  -0.1890 + 0.5817i  -0.1623 + 0.4995i   0.4948 + 0.3595i
  -0.3246 - 0.0000i  -0.3780 + 0.0000i  -0.3246 + 0.0000i   0.9896 - 0.0000i
  -0.4249 + 0.3087i  -0.0722 - 0.2222i   0.5252 - 0.0000i   0.6116 - 0.0000i
  -0.4249 - 0.3087i   0.1890 - 0.1373i   0.1623 + 0.4995i  -0.4948 + 0.3595i
   0.1003 - 0.3087i  -0.3058 + 0.2222i   0.3246            -0.9896 - 0.0000i
   0.5252 + 0.0000i  -0.4948 - 0.3595i   0.1623 + 0.4995i  -0.4948 + 0.3595i

  Columns 5 through 6

   0.4948 + 0.3595i   0.4249 - 0.3087i
   0.9896 - 0.0000i   0.8498 - 0.0000i
   0.1890 - 0.5817i   1.1125 + 0.8082i
   0.1890 + 0.5817i  -0.4249 - 1.3078i
   0.8006 + 0.5817i  -0.2626 - 0.8082i
   0.6116 - 0.0000i   0.1623 - 0.4995i

bper =

  Columns 1 through 4

  -0.6875 + 0.4995i  -0.8006 + 0.5817i  -0.6875 + 0.4995i  -0.1168 + 0.3595i
   0.1003 - 0.3087i   0.1168 - 0.3595i   0.1003 - 0.3087i   0.8006 + 0.5817i
  -0.6875 - 0.4995i   0.3058 - 0.2222i   0.2626 + 0.8082i   0.3058 - 0.2222i
  -0.1623 - 0.4995i   0.2336 + 0.0000i  -0.1623 + 0.4995i   0.4948 + 0.3595i
   0.4249 - 0.3087i  -0.6116 - 0.0000i   0.4249 + 0.3087i  -0.1890 - 0.5817i
  -0.1623 + 0.4995i   0.4948 - 0.3595i  -0.5252            -0.6116          

  Columns 5 through 6

  -0.1168 + 0.3595i  -0.1003 - 0.3087i
   0.8006 + 0.5817i   0.6875 - 0.4995i
  -0.1168 - 0.3595i   0.2626 + 0.8082i
   0.4948 - 0.3595i   1.3751 - 0.0000i
  -0.1890 + 0.5817i  -0.5252 + 0.0000i
   0.4948 + 0.3595i  -0.1623 - 0.4995i    
\end{verbatim}
the matrices
\begin{verbatim}
H =

     1     0     0     0     0     0
     0     1     0     0     0     0
     0     0     1     0     0     0
     0     0     0     1     0     0
     0     0     0     0     1     0
     0     0     0     0     0     1


Q =

     1     0     0     0     0     0
     1     0     0     1     0     0
     0     0     0     0     1     0
     0     0     1     0     1     0
     0     1     1     0     1     0
     0     0     1     0     1     1.
\end{verbatim}
This implies that there is one real oval as can be also seen in 
Fig.~\ref{figfermat}, and the curve is non-dividing. This corresponds 
to $(g,k,a)=(6,1,1)$. 
\begin{figure}[htb!]
\begin{center}
\includegraphics[width=0.6\textwidth]{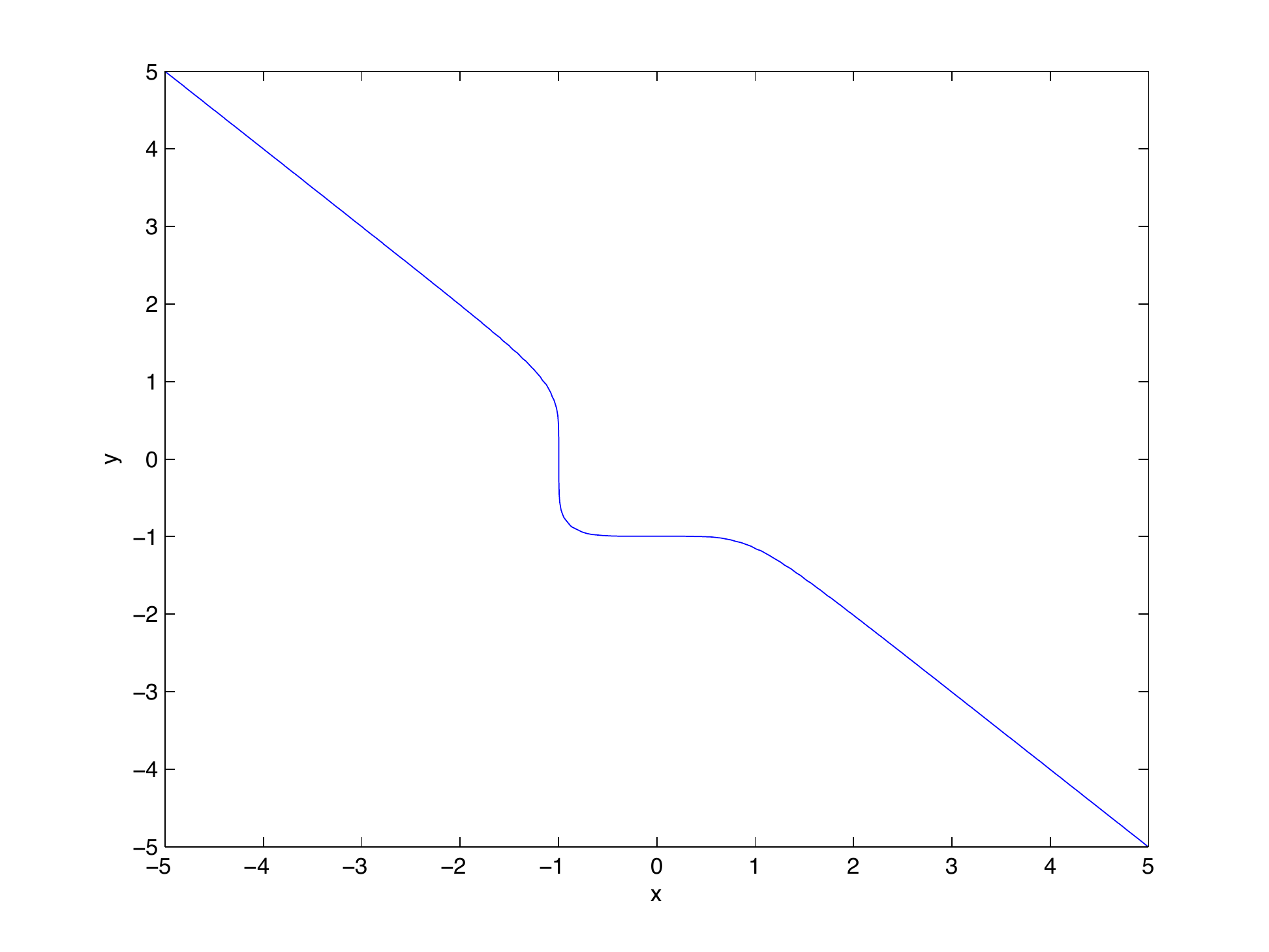}
\end{center}
 \caption{Real variety of the Fermat curve (\ref{fermat}) for $n=5$.}
 \label{figfermat}
\end{figure}

For the curve 
\begin{equation}
    y^3-2x^3y-x^9=0
    \label{x9}
\end{equation}
of genus 3 we have 
\begin{verbatim}
aper =

   0.4021 - 0.6964i  -0.6748 - 1.1688i   0.5985 + 1.0367i
  -0.4764 - 0.9006i   1.5026 - 0.1823i   2.0418 + 0.5711i
  -1.5598 - 0.9006i   0.3157 - 0.1823i  -0.9892 + 0.5711i

bper =

   1.1577 + 0.2041i   0.3591 - 0.9865i   0.3907 + 0.4656i
   0.5417 - 0.3128i   0.5934 + 0.3426i   1.5155 + 0.8750i
   0.6160 + 0.1086i  -0.2344 + 0.6439i  -1.1249 - 1.3406i
    
\end{verbatim}
which leads to 
\begin{verbatim}
H =

     1     0     0
     0     1     0
     0     0     1


Q =

     1     1     1
     0     0     1
     0     1     0.    
\end{verbatim}
This gives the topological type $(3,1,1)$. In particular, the number 
of real ovals equals one.  The real variety of the curve, which has a 
self intersection and a cusp, can be seen in  Fig.~\ref{figx9}. 
\begin{figure}[htb!]
\begin{center}
\includegraphics[width=0.6\textwidth]{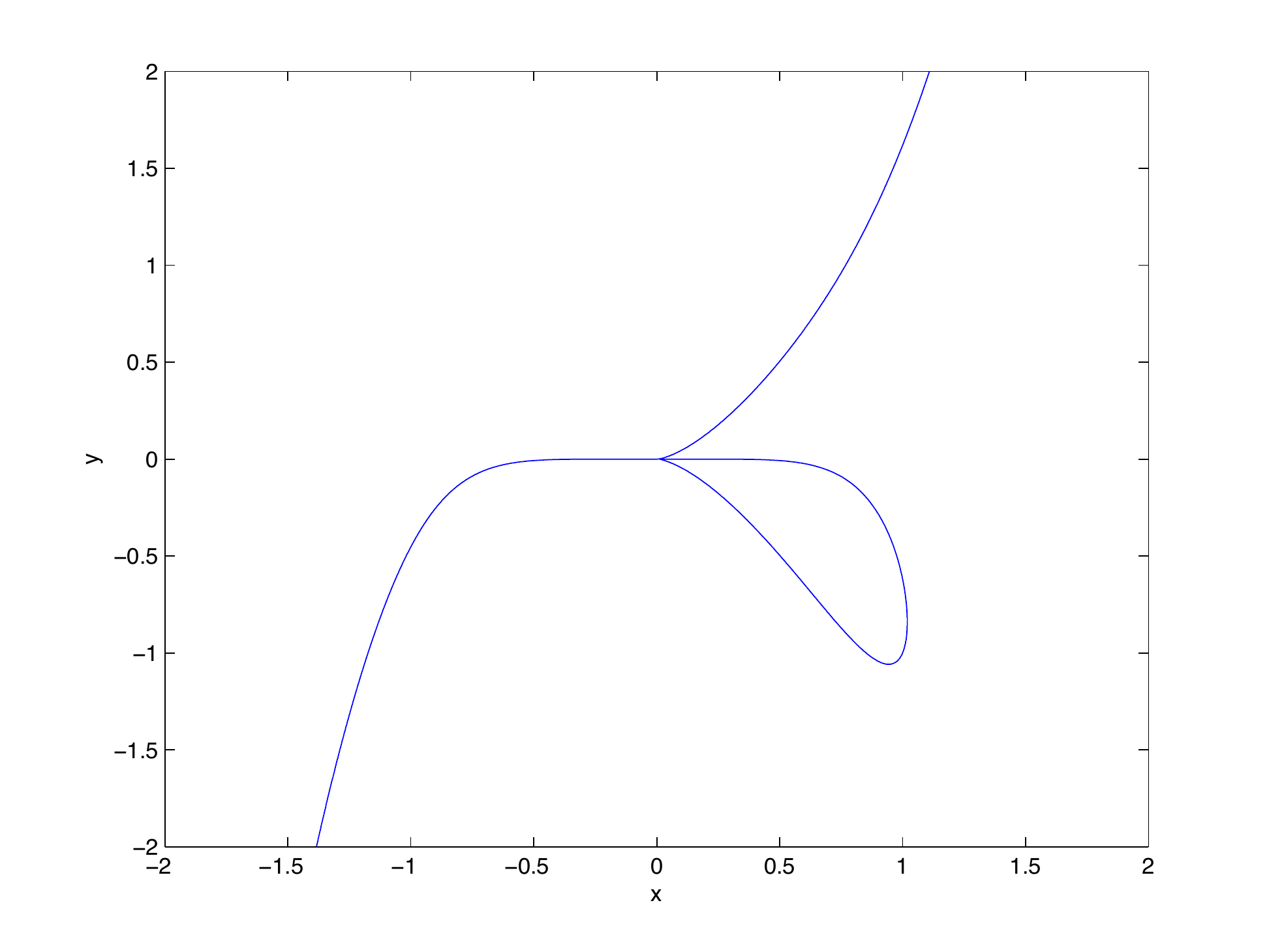}
\end{center}
 \caption{Real variety of the curve (\ref{x9}).}
 \label{figx9}
\end{figure}

The curve 
\begin{equation}
    -180x^5+396yx^4-307x^3y^2+107x^2y^3+273x^3-318x^2y-17xy^4+117xy^2-68x+y^5-12y^3+19y=0
    \label{dividing}
\end{equation}
is known (see \cite{Dubm}) to be a dividing curve of genus 6. In 
fact we get for 
\begin{verbatim}
aper =

  Columns 1 through 4

   0.0414 + 0.0278i  -0.0345 + 0.0272i  -0.0979 + 0.0264i  -0.3041 + 0.0603i
  -0.1149 - 0.0446i   0.0000 - 0.0000i   0.0532 - 0.0226i   0.0000 - 0.0000i
   0.1149 + 0.0169i  -0.0000 + 0.0544i  -0.0532 - 0.0805i  -0.0000 + 0.1206i
   0.0000 - 0.0111i   0.0000 + 0.0183i   0.0000 - 0.0303i   0.0000 + 0.1114i
   0.0820 - 0.0278i   0.0000 + 0.0000i  -0.0527 - 0.1031i   0.0000 + 0.0000i
  -0.0820 - 0.0000i  -0.0000 - 0.0544i   0.0527 - 0.0000i  -0.0000 - 0.1206i

  Columns 5 through 6

  -0.5369 + 0.0872i  -2.8149 + 0.5433i
   0.2427 + 0.0099i   0.7641 + 0.0357i
  -0.2427 - 0.2372i  -0.7641 - 0.5381i
   0.0000 - 0.1843i   0.0000 - 1.1224i
  -0.0881 - 0.2274i  -0.1275 - 0.5024i
   0.0881 - 0.0000i   0.1275 - 0.0000i

bper =

  Columns 1 through 4

   0.0414 + 0.0278i  -0.0345 - 0.0094i  -0.0979 + 0.0264i  -0.3041 - 0.1626i
  -0.0089 - 0.1009i  -0.0666 + 0.0180i   0.1091 - 0.0733i  -0.1162 + 0.0046i
  -0.0089 - 0.0563i  -0.0666 - 0.0180i   0.1091 - 0.0507i  -0.1162 - 0.0046i
   0.0320 - 0.0000i   0.0000 - 0.0183i   0.1425 - 0.0000i   0.0000 - 0.1114i
   0.1060 + 0.0286i   0.0666 + 0.0724i   0.0559 - 0.0525i   0.1162 + 0.1252i
  -0.0580 + 0.0160i   0.0666 - 0.0363i   0.1614 + 0.0751i   0.1162 - 0.1160i

  Columns 5 through 6

  -0.5369 + 0.0872i  -2.8149 + 0.5433i
   0.3380 - 0.1055i   0.9555 - 0.2619i
   0.3380 - 0.1154i   0.9555 - 0.2976i
   0.8311 - 0.0000i   4.8657 - 0.0000i
   0.0954 - 0.1120i   0.1914 - 0.2048i
   0.2716 + 0.1021i   0.4464 + 0.1691i
\end{verbatim}
the matrices
\begin{verbatim}
    H =
    
         0     1     0     0     0     0
         1     0     0     0     0     0
         0     0     0     1     0     0
         0     0     1     0     0     0
         0     0     0     0     0     0
         0     0     0     0     0     0
    
    
    Q =
    
         1     0     0     0     0     0
         0     1     0     0     0     0
         1     1     1     0     0     0
         0     1     0     0     1     0
         0     0     0     1     0     0
         0     0     0     0     1     1.
\end{verbatim}
This means that it is a dividing curve with 3 real ovals, which is, 
however, not obvious from Fig.~\ref{figdivid}. The 
topological type is $(6,3,0)$. 
\begin{figure}[htb!]
\begin{center}
\includegraphics[width=0.6\textwidth]{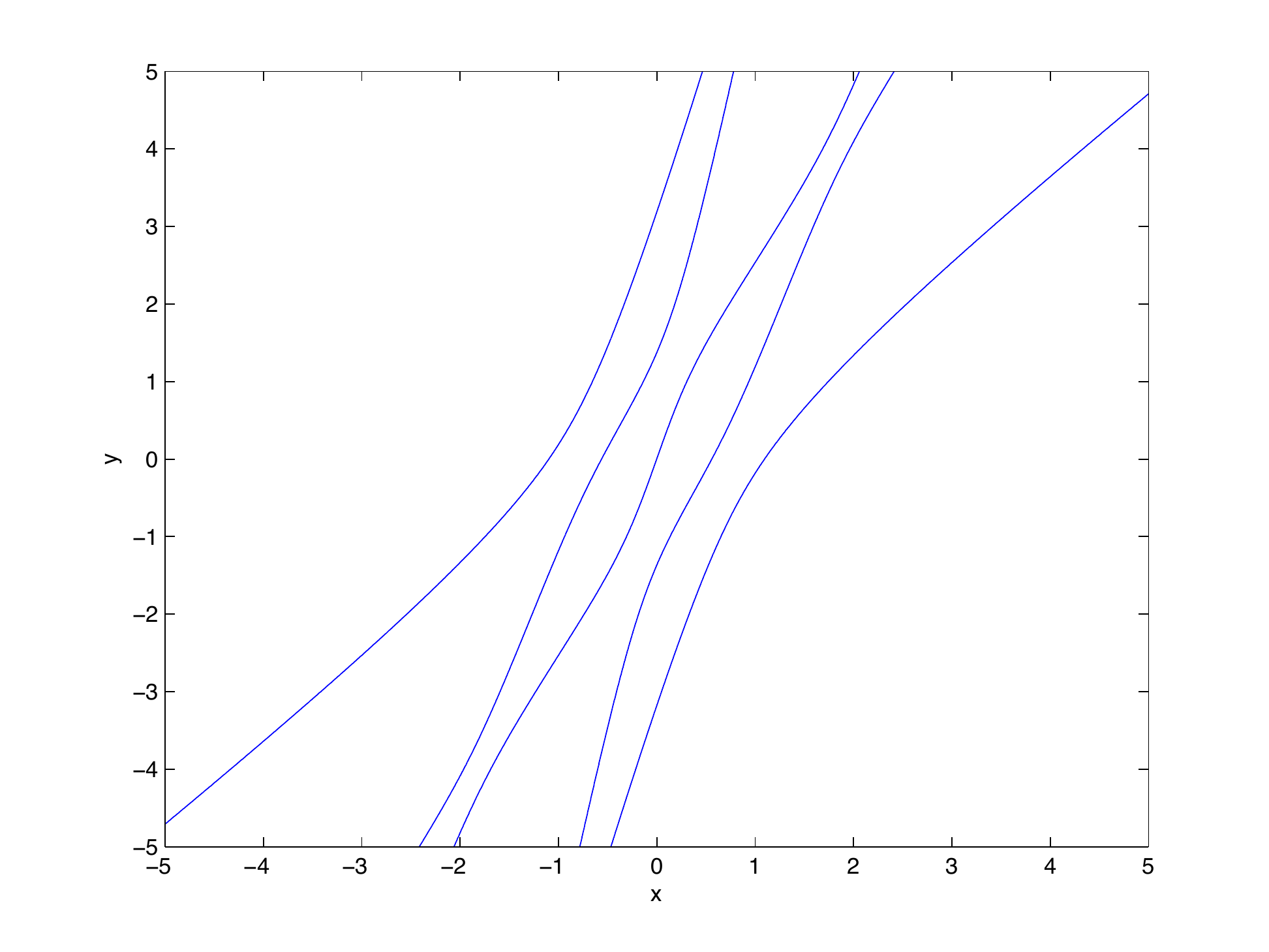}
\end{center}
\caption{Real variety of the curve (\ref{dividing}).}
 \label{figdivid}
\end{figure}

% \[\left[\left(y^2+P(x)-P(\zeta x)-P(\zeta^{2} x)\right)^{2}-4\,y^{2}P(x)-4\,P(\zeta x)P(\zeta^{2} x)\right]^{2}-64\,y^{2}P(x)P(\zeta x)P(\zeta^{2} x)=0\]
% where $P(x)=(x-1)(\zeta x-1)(\zeta^{2} x-1)(\zeta^{4} x-1)$ and $\zeta=\exp\{\frac{2\mathrm{i}\pi}{7}\}$.

\section{Outlook}
The algorithm presented in this paper allows the transformation of an 
arbitrary canonical homology basis to a form adapted to the 
underlying symmetry, here the anti-holomorphic involution. This 
permits to find a basis satisfying relation (\ref{hom basis}). A 
similar condition can be imposed for any involution,
and the algorithm presented here can be easily adapted 
to that case. 

In general, the presence of symmetries allows to significantly 
simplify the Riemann matrix of a surface, but only in a homology 
basis adapted to the symmetries, for instance in a basis such that 
the $A$-cycles are invariant under symmetry operations, see 
\cite{BBEIM} and  \cite{braden} for the Klein curve. The latter 
reference uses an approach to this 
problem based on Comessati's theorem \cite{Com} via two pieces of 
software,     \emph{ extcurves}
 and \emph{CyclePainter}\footnote{Located at http://gitorious.org/riemanncycles.}. It will be the subject of further work to 
generalize the algorithm to symmetry groups beyond involutions.  
In a first step it would be interesting to extend 
the approach presented in this paper to automorphisms $\tau$ satisfying 
$\tau^{n}=id$ with $n>2$.


\begin{thebibliography}{99}

\bibitem{Arnold} V. I. Arnold, \textit{The situation of ovals of real plane algebraic curves, the involutions
of four-dimensional smooth manifolds, and the arithmetic of integral quadratic forms},
Funkcional. Anal. i Prilozen \textbf{5}, 1--9  (1971).
\bibitem{Arnon} D. S. Arnon and S. McCallum, \textit{A polynomial time
algorithm for the topological type of a real algebraic curve}, Journal of Symbolic Computation, 5:213--236 (1988).
\bibitem{BBEIM} E. Belokolos, A. Bobenko, V. Enolskii, A. Its, V. Matveev, \textit{Algebro-geometric approach to nonlinear integrable equations}, Springer Series in nonlinear dynamics (1994).
 \bibitem{Bob} A.I. Bobenko,  C. Klein, (ed.), \textit{Computational Approach to Riemann 
Surfaces}, Lect. Notes Math. \textbf{2013} (2011).
\bibitem{braden} H. W. Braden, T.P. Northover, \textit{Klein's Curve}, J. Phys. A \textbf{43}, 434009 (2010).
\bibitem{bradengraphical}H. Braden, V.~Enolskii, T.~Northover, Maple packages 
\emph{extcurves} and \emph{CyclePainter}, available at gitorious.org.
\bibitem{Com}A.~Comessati, \emph{Sulla connessione delle superficie 
algebriche reali}, Annali di Mat. (3) 23, 215--283 (1915).
\bibitem{Coste} M. Coste, M.-F. Roy, \textit{Thom's lemma, the coding of real algebraic numbers and the computation of the
topology of semi-algebraic sets}, Journal of Symbolic Computation, 5:121--129 (1988).
%\bibitem{DS} A. Davey and K. Stewartson, {\it On three-dimensional packets of surface waves}, Proc. R. Soc. Lond. A {\bf 388}, 101--110 (1974).
\bibitem{deco1} B.~Deconinck, M.~van Hoeij, \textit{Computing Riemann matrices of algebraic curves}, Physica	D, \textbf{28}, 152--153 (2001).	 
\bibitem{deco2} B.~Deconinck, M.~Heil, A.~Bobenko, M.~van	 Hoeij, M.~Schmies, \textit{Computing Riemann Theta Functions},  	 Mathematics of Computation	 \textbf{73},  1417  (2004).   
\bibitem{deho} B.~Deconinck, M.~Patterson, \textit{Computing with 
plane algebraic curves and Riemann surfaces: the algorithms of the 
Maple package ``algcurves''}, in A.I.~Bobenko,  C.~Klein, (ed.), \textit{Computational Approach to Riemann 
Surfaces}, Lect. Notes Math. \textbf{2013} (2011).
\bibitem{Dubm} B.A. Dubrovin, \textit{Matrix finite-zone operators}, Revs. Sci. Tech. \textbf{23}, 20--50 (1983).
\bibitem{DN} B. Dubrovin, S. Natanzon, {\it Real theta-function 
solutions of the Kadomtsev-Petviashvili equation}, Math. USSR Ivestiya, \textbf{32}:2, 269--288 (1989).
%\bibitem{Fay} J. Fay, {\it Theta functions on Riemann surfaces}, Lecture Notes in Mathematics \textbf{352} (1973).
\bibitem{Feng} H. Feng, \textit{Decomposition and Computation of the 
Topology of Plane Real Algebraic Curves}, Ph.D. thesis, The Royal
Institute of Technology, Stockholm (1992).
\bibitem{FK1}     J.~Frauendiener, C.~Klein, \textit{Hyperelliptic theta functions 	and spectral methods}, J. Comp. Appl. Math. (2004).    
\bibitem{FK2} J.~Frauendiener, C.~Klein, \textit{Hyperelliptic theta functions	and spectral methods: KdV and KP solutions},		Lett. Math. Phys., Vol. \textbf{76}, 249--267 (2006). 
\bibitem{FK}  J.~Frauendiener, C.~Klein,  \textit{Algebraic curves and Riemann surfaces in Matlab}, in A.~Bobenko and C.~Klein (ed.), \textit{Riemann Surfaces--Computational Approaches}, Lecture Notes in Mathematics Vol. \textbf{2013} (Springer) (2011).
\bibitem{Gab} A. Gabard, \textit{Sur la topologie et la géométrie des courbes algébriques réelles}, Ph.D. Thesis (2004).
\bibitem{Gonz} L. Gonzalez-Vega, I. Necula, \textit{Efficient topology determination of implicitly defined algebraic plane curves}, 
Computer Aided Geometric Design, 19:719--743 (2002).
\bibitem{Gud} D.A. Gudkov, \textit{Complete topological classification of the disposition of ovals of a sixth order curve
in the projective plane}, Gor'kov. Gos. Univ. Ucen. Zap. Vyp., 87:118--153 (1969).
\bibitem{Harnack} A. Harnack, \textit{Ueber die Vieltheiligkeit der ebenen algebraischen Curven}, Math. Ann., \textbf{10}, 189--199 (1876).
\bibitem{HafCur}J. Hafner and K. McCurley, Asymptotically fast 
triangularization of matrices over rings, SIAM J. Comput. 20, 1068-1083  (1991).
\bibitem{Hilbert1} D. Hilbert, \textit{Über die reellen Züge algebraischen Kurven}, Math. Ann., \textbf{38}, 115--138 (1891).
\bibitem{Hilbert2} D. Hilbert, \textit{Mathematische Probleme}, Arch. Math. Phys., 1:43--63, (German) (1901).
\bibitem{Hong} H. Hong, \textit{An efficient method for analyzing the topology of plane real algebraic curves}, Mathematics
and Computers in Simulation, 42:571--582 (1996).
\bibitem{Kalla} C. Kalla, \textit{New degeneration of Fay's identity 
and its application to integrable systems}, \emph{preprint } arXiv:1104.2568v1 (2011).
\bibitem{KKnum} C.~Kalla, C. Klein, \emph{On the numerical evaluation of algebro-geometric solutions  to integrable  equations}, Nonlinearity \textbf{25} 569-596  (2012).
\bibitem{Klein} F. Klein,  \textit{On Riemann's theory of algebraic functions and their integrals}, Dover (1963).
\bibitem{Lang} S. Lang, \textit{Algebra}, Revised Third Edition, Springer-Verlag, New York (2002).
\bibitem{Mal} T. Malanyuk, {\it Finite-gap solutions of the Davey-Stewartson equations}, J. Nonlinear Sci, \textbf{4}, No. 1, 1--21 (1994).
%\bibitem{Mum} D. Mumford, {\it Tata Lectures on Theta. I and II.}, Progress in Mathematics, 28 and 43, respectively. Birkhäuser Boston, Inc., Boston, MA, 1983 and 1984.
\bibitem{Pet} I. Petrovsky, \textit{On the topology of real plane algebraic curves}, Ann. Math., \textbf{39}, No. 1, 187--209
(1938).
\bibitem{Rohn} K. Rohn, \textit{Die Maximalzahl und Anordnung der 
Ovale bei der ebenen Kurve 6. Ordnung und bei der Fl\"ache
4. Ordnung}, Math. Ann., \textbf{73}, 177-229 (1913).
\bibitem{Sak} T. Sakkalis, \textit{The topological configuration of a real algebraic curve}, Bulletin of the Australian
Mathematical Society, 43:37--50 (1991).
\bibitem{Seidel} R. Seidel, N. Wolpert, \textit{On the exact computation of the
topology of real algebraic curves}, In Proc 21st ACM Symposium on Computational Geometry, p 107--115 (2005).
\bibitem{sesi}M. Sepp\"a\l\"a and R. Silhol,\textit{ Moduli Spaces for Real 
Algebraic Curves and Real Abelian Varieties}, Math. Z. \textbf{201}, 
151-165 (1989).  
%\bibitem{tref} L.N. Trefethen, \textit{Spectral Methods in  Matlab}, SIAM, Philadelphia, PA (2000).
\bibitem{tretalg} C.L. Tretkoff, M.D.  Tretkoff,  \textit{Combinatorial group theory, Riemann surfaces and differential equations},   Contemporary Mathematics, \textbf{33}, 467--517 (1984).
\bibitem{Trott} M.~Trott, \emph{Applying Groebner Basis to Three Problems in 
Geometry}, Mathematica in Education and Research 6 (1): 15--28 (1997).
\bibitem{Vin} V. Vinnikov, {\it Self-adjoint determinantal representations of real plane curves}, Math. Ann. \textbf{296},
453--479 (1993).
\bibitem{Viro} O. Ya. Viro, \textit{Gluing of plane real algebraic curves and constructions of curves of degrees 6 and
7}, Topology (Leningrad, 1982), Lecture Notes in Math., vol. \textbf{1060}, Springer, Berlin, pp. 187--200 (1984).






\end{thebibliography}
\end{document}